\documentclass[12pt]{article}
\usepackage[margin=2cm]{geometry} 
\usepackage{etoolbox}
\newtoggle{siads} 
\togglefalse{siads}

\newtoggle{public} 
\togglefalse{public}
\toggletrue{public}

\iftoggle{siads}
{
\usepackage{cite}
}
{
\usepackage[                   
  backend=bibtex8,             
  bibencoding=utf8,            
  style=alphabetic,            
  natbib=true,                 
  hyperref=true,               
  backref=true,                
  isbn=false,                  
  url=true,                    
  doi=true,                    
  urldate=long,                
  firstinits=true,
  terseinits=false,
  maxnames=3,%
  minnames=1,%
  maxbibnames=25,%
  minbibnames=15,%
  maxcitenames=2,%
  mincitenames=1%
]{biblatex}

\setcounter{biburlnumpenalty}{100}

\renewbibmacro*{doi+eprint+url}{%
  \iftoggle{bbx:doi}
    {\printfield{doi}}
    {}%
  \newunit\newblock
  \iftoggle{bbx:eprint}
    {\iffieldundef{doi}{\iffieldundef{url}{\printfield{eprint}}{}}{}}
    {}%
  \newunit\newblock
  \iftoggle{bbx:url}
    {\iffieldundef{doi}{\printfield{url}}{}}
    {}}
}

\usepackage[
    hypertexnames,%
    citecolor=blue,%
    colorlinks=true,%
    linkcolor=red%
]{hyperref}
\usepackage[T1]{fontenc}
\usepackage[all]{hypcap}
\usepackage{enumitem}

\usepackage[font=footnotesize]{caption}
\usepackage[font=footnotesize,caption=false]{subfig}
\iftoggle{siads}{}{\usepackage{amsthm}}
\usepackage{amssymb,amsmath}
\usepackage{dsfont} 
\usepackage[mathscr]{eucal}
\usepackage{verbatim}

\usepackage{color}
\usepackage{graphicx}

\usepackage{url}

\newcommand{\fig}[1]{{Figure~\ref{fig:#1}}}
\newcommand{\sct}[1]{{Section~\ref{sec:#1}}}
\newcommand{\app}[1]{{Appendix~\ref{app:#1}}}

\newcommand{\ECr}{event--selected $C^r$}
\newcommand{\ECrD}{\ensuremath{EC^r(D)}}
\newcommand{\EC}[1]{event--selected $C^{#1}$}
\newcommand{\BC}[2]{\ensuremath{B^{C^{#1}}_{#2}}}
\newcommand{\localsec}{local section}

\newcommand{\see}[1]{(see~#1)}
\newcommand{\set}[1]{\left\{ #1 \right\}}
\newcommand{\paren}[1]{\left( #1 \right)}
\newcommand{\brak}[1]{\left[ #1 \right]}

\newcommand{\abs}[1]{\left| #1 \right|}

\newcommand{\norm}[1]{\left\| #1 \right\|}
\newcommand{\norma}[2]{\left\| #2 \right\|_{#1}}
\newcommand{\normo}[1]{\left\| #1 \right\|_1}

\newcommand{\inorm}[1]{\norma{i}{#1}}
\newcommand{\pw}[1]{\left\{\begin{array}{ll} #1 \end{array}\right. }
\newcommand{\mat}[2]{\brak{\begin{array}{#1} #2 \end{array}}}

\newcommand{\bd}{\partial}
\newcommand{\st}{\mid}
\newcommand{\cube}[1]{B_{#1}}
\newcommand{\cuben}{\cube{n}}
\newcommand{\Ccoprod}[1]{C^r\paren{\coprod_{b\in\cuben} #1,\coprod_{b\in\cuben} T#1}}
\newcommand{\pmones}{\set{-1,+1}^n}

\newcommand{\td}[1]{\widetilde{#1}}
\newcommand{\ha}[1]{\widehat{#1}}
\newcommand{\obar}[1]{\overline{#1}}

\newcommand{\eqn}[1]{\begin{equation*}\begin{aligned} #1 \end{aligned}\end{equation*}}
\newcommand{\eqnn}[1]{\begin{equation}\begin{aligned} #1 \end{aligned}\end{equation}}

\newcommand{\sm}{\setminus}
\newcommand{\into}{\rightarrow}
\newcommand{\goesto}{\rightarrow}

\newcommand{\R}{\mathbb{R}}
\newcommand{\Z}{\mathbb{Z}}

\newcommand{\N}{\mathbb{N}}
\newcommand{\T}{\top}

\newcommand{\e}{\mathscr}

\newcommand{\vphi}{\varphi}
\newcommand{\veps}{\varepsilon}

\newcommand{\ones}{\mathds{1}}
\newcommand{\zerod}{0_d}
\newcommand{\zerodd}{0_{2d}}
\newcommand{\zerodxd}{0_{d\times d}}
\newcommand{\fp}{\rho}
\newcommand{\flbx}{\chi}
\newcommand{\ipct}{\psi}

\newcommand{\pre}[1]{#1^{-1}}
\newcommand{\id}{\operatorname{id}}

\newcommand{\conv}{\operatorname{conv}}
\newcommand{\Int}[1]{\operatorname{Int}#1}
\newcommand{\co}[1]{\operatorname{co}#1}

\newcommand{\pset}[1]{2^{#1}}

\newcommand{\sgn}{\operatorname{sign}}
\newcommand{\word}{\omega}
\newcommand{\Words}{\Omega}

\newcommand{\Pmap}{Poincar\'{e} map}

\iftoggle{public}{
\newcommand{\tbd}[1]{}
\newcommand{\todo}[1]{}
\newcommand{\sam}[1]{}
\newcommand{\kod}[1]{}
\newcommand{\SR}[1]{}
}
{
\newcommand{\tbd}[1]{{\normalsize{\textsc{({\color{red}TBD:\ }\color{blue}#1)}}}}
\newcommand{\todo}[1]{{\normalsize{{({\color{red}TODO:\ }\color{blue}#1)}}}}
\newcommand{\sam}[1]{{\normalsize{{({Sam:\ }\color{blue}#1})}}}
\newcommand{\kod}[1]{{\normalsize{{({Kod:\ }\color{red}#1})}}}
\newcommand{\SR}[1]{{\normalsize{\textsc{({SR:\ }\color{magenta}#1})}}}
}

\iftoggle{siads}{}{
\newtheorem{proposition}{Proposition}
\newtheorem{definition}{Definition}

\newtheorem{theorem}{Theorem}
\newtheorem{corollary}{Corollary}
\newtheorem{lemma}{Lemma}
}
\newtheorem{claim}{Claim}
\newtheorem{assumption}{Assumption}
\newtheorem{remark}{Remark}
\newtheorem{example}{Example}

\newcommand{\defn}[1]{\begin{definition} #1 \end{definition}}
\newcommand{\rem}[1]{\begin{remark} #1 \end{remark}}

\newcommand{\pf}[1]{\begin{proof} #1 \end{proof}}

\newcommand{\thmflow}{Theorem~\ref{thm:flow} (local flow)}
\newcommand{\corflow}{Corollary~\ref{cor:flow} (global flow)}

\newcommand{\thmimpact}{Theorem~\ref{thm:impact} (time--to--impact)}

\newcommand{\thmflowbox}{Corollary~\ref{thm:flowbox} (flowbox)}
\newcommand{\thmpmap}{Theorem~\ref{thm:pmap} (Poincar\'{e} map)}

\newcommand{\thmpertvf}{Theorem~\ref{thm:ss1} (vector field perturbation)}
\newcommand{\thmpertef}{Theorem~\ref{thm:ss2} (event function perturbation)}
\newcommand{\propstab}{Proposition~\ref{prop:stab} (contractivity test for periodic orbit stability)}
\newcommand{\propinorm}{Proposition~\ref{prop:inorm} (induced norm test for periodic orbit stability)}

\newcommand{\assinc}{Assumption~\ref{ass:filippov} (differential inclusion basic conditions)}
\newcommand{\thminc}{Theorem~\ref{thm:filippov} (differential inclusion perturbation)}

\newcommand{\chainrule}{chain rule~\cite[Theorem~3.1.1]{Scholtes2012}}
\newcommand{\Frechet}{Fr{\'{e}}chet}
\newcommand{\Jacobian}{Jacobian}
\newcommand{\Poincare}{Poincar{\'{e}}}
\newcommand{\CMT}{Banach contraction mapping principle~\cite{Banach1922}~\cite[Lemma~C.35]{Lee2012}}

\newcommand{\sect}[1]{\section{#1}}
\newcommand{\subsect}[1]{\subsection{#1}}
\newcommand{\subsubsect}[1]{\subsubsection{#1}}

\usepackage{lineno}

\usepackage{environ}
\NewEnviron{killcontents}{}

\title{Event--Selected Vector Field Discontinuities Yield Piecewise--Differentiable Flows}

\author{
Samuel~A.~Burden%
\thanks{EE Dept., University of Washington, Seattle, WA, USA ({\tt sburden@uw.edu}).}
\and S.~Shankar~Sastry%
\thanks{EECS Dept., University of California, Berkeley, CA, USA ({\tt sastry@eecs.berkeley.edu}).}
\and  Daniel~E.~Koditschek%
\thanks{ESE Dept., University of Pennsylvania, Philadelphia, PA, USA ({\tt kod@seas.upenn.edu})}%
\and Shai~Revzen%
\thanks{EECS Dept., University of Michigan, Ann Arbor, MI, USA ({\tt shrevzen@eecs.umich.edu})}%
}

\date{}

\iftoggle{siads}
{
}
{
\addbibresource{multi.bib}
}

\begin{document}

\maketitle

\iftoggle{siads}{}{\tableofcontents}

\iftoggle{siads}{}{\iftoggle{public}{}{\linenumbers}}
\begin{abstract}
We study a class of discontinuous vector fields brought to our attention by multi--legged animal locomotion. 
Such vector fields arise not only in biomechanics, but also in robotics, neuroscience, and electrical engineering, to name a few domains of application.
Under the conditions that (i) the vector field's discontinuities are locally confined to a finite number of smooth submanifolds and (ii) the vector field is transverse to these surfaces in an appropriate sense, we show that the vector field yields a well--defined flow that is Lipschitz continuous and piecewise--differentiable. 
This implies that although the flow is not classically differentiable, nevertheless it admits a first--order approximation (known as a Bouligand derivative) that is piecewise--linear and continuous at every point. 
We exploit this first--order approximation to infer existence of piecewise--differentiable impact maps (including Poincar\'{e} maps for periodic orbits), 
show the flow is locally conjugate (via a piecewise--differentiable homeomorphism) to a flowbox,
and assess the effect of perturbations (both infinitesimal and non--infinitesimal) on the flow.
We use these results to give a sufficient condition for the exponential stability of a periodic orbit passing through a point of multiply intersecting events, and apply the theory in illustrative examples to demonstrate synchronization in abstract first-- and second--order phase oscillator models.

\end{abstract}


%
%


\sect{Introduction}\label{sec:intro}

We study a class of discontinuous vector fields brought to our attention by multi--legged animal locomotion. 
%
Parsimonious dynamical models for diverse physical phenomena are
governed by vector fields that are smooth except along a finite number
of surfaces of discontinuity.  Examples include: integrate--and--fire
neurons that undergo a discontinuous change in membrane voltage during
a threshold crossing~\cite{KeenerHoppensteadt1981, HopfieldHerz1995,
BizzarriBrambilla2013}; legged locomotors that encounter
discontinuities in net forces due to intermittent interaction of
viscoelastic limbs with terrain~\cite{Alexander1984,
GolubitskyStewart1999, HolmesFull2006}; electrical power systems that
undergo discontinuous changes in network topology triggered by
excessive voltages or currents~\cite{Hiskens1995}.  In each of these
examples, behaviors of interest---%
\emph{phase locking}~\cite{KeenerHoppensteadt1981}
or \emph{local synchronization}~\cite{HopfieldHerz1995};
simultaneous touchdown of two or more legs~\cite{Alexander1984,
GolubitskyStewart1999, HolmesFull2006}; 
\emph{voltage collapse}
phenomena~\cite{DobsonLu1992}~\cite[Section~II-A.2]{Hiskens1995}---%
occur at or near the intersection of multiple surfaces of
discontinuity.  Although analytical tools exist to study orbits that
pass transversely through non--intersecting switching surfaces
(e.g. to assess stability~\cite{AizermanGantmacher1958,
GrizzleAbba2002}, compute first--order
variations~\cite{HiskensPai2000, WendelAmes2012}, and reduce
dimensionality~\cite{BurdenRevzen2015tac}), piecewise--defined (or \emph{hybrid}) systems that admit
simultaneous discrete transitions generally exhibit ``branching''
wherein the flow depends discontinuously%
\footnote{
We note that hybrid state spaces do not possess a natural metric, and
continuity of the flow depends on the chosen metric; this issue is
discussed in detail
elsewhere~\cite[Sec.~V-A]{BurdenGonzalezVasudevan2015tac}.} 
 on initial conditions~\cite[Definition~3.11]{SimicJohansson2005}.
For instance, in the mechanical setting, the flow of a Lagrangian
dynamical system subject to unilateral constraints is generically
discontinuous near simultaneous--impact
events~\cite[Section~7]{Ballard2000}.  In the case where a vector
field is discontinuous across two transversally--intersecting
surfaces, others have established continuity and derived first--order
approximations of the flow~\cite{Ivanov1998, Di-BernardoBudd2008, DieciLopez2011, BizzarriBrambilla2013}.  Techniques applicable to arbitrary numbers of
surfaces have been derived for the case of pure phase oscillators with
perpendicular transition surfaces~\cite{MirolloStrogatz1990}.  

We generalize these approaches to accommodate an arbitrary number of
nonlinear transition surfaces that are not required to be transverse
and extend a suite of analytical and computational techniques from
classical (smooth) dynamical systems theory to the present
(non--smooth) setting.  Under the conditions that (i) the vector
field's discontinuities are locally confined to a finite collection of
smooth submanifolds and (ii) the vector field is ``transverse'' to
these surfaces in an appropriate sense, we show that the vector field
yields a well--defined flow that is Lipschitz continuous and
piecewise--differentiable.  
The definition of
piecewise--differentiability we employ (introduced
in~\cite{Robinson1987, Rockafellar2003, Scholtes2012}) implies
that although the flow is not classically differentiable, nevertheless
it admits a first--order approximation (the so--called \emph{Bouligand
derivative} or \emph{B--derivative}~\cite[Chapter~3]{Scholtes2012})
that is piecewise--linear and continuous at every point.  
We exploit
this first--order approximation to infer existence of
piecewise--differentiable impact maps (including Poincar\'{e} maps for
periodic orbits), assess the effect of perturbations on the flow,
and derive a straightforward procedure to compute the B--derivative. We
use these results to give a sufficient condition for the exponential stability of
a periodic orbit passing through a point of multiply intersecting
events, and apply the theory 
in illustrative examples to demonstrate synchronization in
abstract first-- and second--order phase oscillator models.

The paper is organized as follows. 
Following a brief review of relevant technical background in~\sct{prelim}, we define the discontinuous but piecewise--smooth vector fields of interest and show that they yield continuous B--differentiable flows in~\sct{floww}.
In~\sct{impact} we demonstrate that such flows are continuously conjugate to classical flows, leading to results in~\sct{pert} establishing their persistence under small perturbations. 
\sct{comp} develops stability results and their application to simple oscillator models is given in~\sct{app}.
The paper concludes with a brief summary in~\sct{disc} suggesting the relevance  of these results to biological and engineered systems of practical interest.


\sect{Preliminaries}\label{sec:prelim}
The mathematical constructions we use are ``standard'' in the sense that they are familiar to practitioners of (applied) dynamical systems or optimization theory (or both), but since this paper represents (to the best of our knowledge) the first application of some techniques from non--smooth analysis to the present class of dynamical systems, the reader may be unfamiliar with some of the more recently--developed devices we employ.
Thus in this section we briefly review mathematical concepts and introduce notation that will be used to state and prove results throughout this paper, and suggest textbook references where the interested reader could obtain a complete exposition.

\subsection{Notation}\label{sec:nota}
To simplify the statement of our definitions and results, we fix notation of some objects in $\R^n$:
$+\ones\in\R^n$ denotes the vector of all ones and $-\ones$ its negative; 
$e_j$ is the $j$--th standard Euclidean basis vector;
$\cuben := \set{-1,+1}^n\subset\R^n$ is the set of corners of the $n$-dimensional cube.
We let $\sgn:\R^n\into\set{-1,+1}^n$ 
be the vectorized signum function taking its values in the Euclidean cube's corners, i.e.
\eqnn{\label{eqn:sgn}
\forall x\in\R^n, j\in\set{1,\dots,n} : e_j^\T \sgn(x) = \pw{-1,& x_j < 0 \\ +1,& x_j \ge 0}.
}

To fix notation, in the following paragraphs we will briefly recapitulate standard constructions from topology, differential topology, and dynamical systems theory, and refer the reader to~\cite{Lee2012} for details.
If $U\subset X$ is a subset of a topological space, then $\Int U\subset X$ denotes its \emph{interior} and $\bd U$ denotes its \emph{boundary}.
Let $f:X\into Y$ be a map between topological spaces.
If $U\subset X$ then $f|_U:U\into Y$ denotes the \emph{restriction}. 
If $V\subset Y$ then $\pre{f}(V) = \set{x\in X : f(x)\in V}$ denotes the \emph{pre--image of $V$ under $f$}.

Given $C^r$ manifolds $D,N$, we let $C^r(D,N)$ denote the set of $C^r$ functions from $D$ to $N$.
$H\subset D$ is a \emph{$C^r$ codimension-$k$ submanifold} of the $d$-dimensional manifold $D$ if every $x\in H$ has a neighborhood $U\subset D$ over which there exists a $C^r$ diffeomorphism $h:U\into\R^d$ such that 
\eqn{
H\cap U = h^{-1}\paren{\set{y\in\R^d : y_{k+1} = \cdots = y_{d} = 0}}. 
}
If $f\in C^r(D,N)$ then at every $x\in D$ there exists an induced linear map $Df(x):T_x D\into T_{f(x)}N$ called the \emph{pushforward} (in coordinates, $Df(x)$ is the Jacobian linearization of $f$ at $x\in D$) where $T_x D$ denotes the tangent space to the manifold $D$ at the point $x\in D$.
Globally, the pushforward is a $C^{r-1}$ map $Df:TD\into TN$ where $TD$ is the tangent bundle associated with the manifold $D$; 
we recall that $TD$ is naturally a $2d$--dimensional $C^r$ manifold.
When $N = \R$, we will invoke the standard identification $T_y N\simeq\R$ for all $y\in N$ and regard $Df(x)$ as a linear map from $T_x D$ (i.e. an element of the \emph{cotangent space} $T_x^* D$) into $\R$ for every $x\in D$;
we recall that the \emph{cotangent bundle} $T^*D$ is naturally a $2d$--dimensional $C^r$ manifold.
If $U\subset D$ and $f:U\into N$ is a map, then a map $\td{f}:D\into N$ is a \emph{$C^r$ extension of $f$} if $\td{f}$ is $C^r$ and $\td{f}|_U = f$.

Following~\cite[Chapter~8]{Lee2012}, a (possibly discontinuous or non--differentiable) map $F:D\into TD$ is a \emph{(rough)}%
\footnote{We will constrain the class of vector fields under consideration in \sct{assump}, but for expediency drop the \emph{rough} modifier in the sequel.}
\emph{vector field} if $\pi\circ F = \id_D$ where $\pi:TD\into D$ is the natural projection and $\id_D$ is the identity map on $D$. 
A vector field may, under appropriate conditions, yield an associated \emph{flow} $\phi:\e{F}\into D$ defined over an open subset $\e{F}\subset \R\times D$ called a \emph{flow domain}; in this case for every $x\in D$ the set $\e{F}^x = \e{F}\cap\paren{\R\times\set{x}}$ is an open interval containing the origin, the restriction $\phi|_{\e{F}^x}:\e{F}^x\into D$ is absolutely continuous, and the derivative with respect to time is $D_t\phi(t,x) = F(\phi(t,x))$ for almost every $t\in\e{F}^x$.
A flow is \emph{maximal} if it cannot be extended to a larger flow domain.
An \emph{integral curve} for $F$ is an absolutely continuous function $\xi:I\into D$ over an open interval $I\subset\R$ such that $\dot{\xi}(t) = F(\xi(t))$ for almost all $t\in I$; it is \emph{maximal} if it cannot be extended to an integral curve on a larger open interval.

\subsection{Piecewise Differentiable Functions and Nonsmooth Analysis}
\label{sec:pcrnsa}

The notion of piecewise--differentiability we employ was originally introduced by Robinson~\cite{Robinson1987}; since the recent monograph from Scholtes~\cite{Scholtes2012} provides a more comprehensive exposition, we adopt the notational conventions therein.
Let $r\in\N\cup\set{\infty}$ and $D\subset\R^d$ be open.
A continuous function $f:D\into\R^n$ is called \emph{piecewise--$C^r$} if for every $x\in D$ there exists an open set $U\subset D$ containing $x$ and a finite collection $\set{f_{j}:U\into\R^n}_{{j}\in\e{J}}$ of $C^r$--functions such that for all $x\in U$
we have $f(x)\in\set{f_{j}(x)}_{{j}\in\e{J}}$.
The functions $\set{f_{j}}_{{j}\in\e{J}}$ are called \emph{selection functions} for $f|_U$, and $f$ is said to be a \emph{continuous selection} of $\set{f_j}_{j\in\e{J}}$.
A selection function $f_{j}$ is said to be \emph{active} at $x\in U$ if $f(x) = f_{j}(x)$.
We let $PC^r(D,\R^n)$ denote the set of piecewise--$C^r$ functions from $D$ to $\R^n$.
Note that $PC^r$ is closed under composition and pointwise maximum or minimum of a finite collection of functions.
Any $f\in PC^r(D,\R^n)$ is locally Lipschitz continuous, and a Lipschitz constant for $f$ is given by the supremum of the induced norms of the ({\Frechet}) derivatives of the set of selection functions for $f$.
Piecewise--differentiable functions possess a first--order approximation $Df:TD\into T\R^n$ called the \emph{Bouligand derivative} (or B--derivative)~\cite[Chapter~3]{Scholtes2012}; this is the content of Lemma~4.1.3 in~\cite{Scholtes2012}.
We let $Df(x;v)$ denote the B--derivative of $f$ evaluated along the tangent vector $v\in T_x D$.
The B--derivative is positively homogeneous, i.e. $\forall v\in T_x D,\lambda\ge 0 : Df(x;\lambda v) = \lambda Df(x;v)$.



\sect{Local and Global Flow}\label{sec:floww}
In this section we rederive in our present nonsmooth setting the erstwhile familiar fundamental construction associated with a vector field: its flow.
We begin in \sct{assump} by introducing the class of vector fields under consideration, namely, {\ECr} vector fields.
Subsequently in \sct{const} we construct a candidate flow function via composition of piecewise--differentiable functions.
Finally in \sct{flow} we show this candidate function is indeed the flow of the {\ECr} vector field.

\subsection{Event--Selected Vector Fields Discontinuities}\label{sec:assump}

The flow of a discontinuous vector field $F:D\into TD$ over an open domain $D\subset\R^d$ can exhibit pathological behaviors ranging from nondeterminism to discontinuous dependence on initial conditions.
We will investigate local properties of the flow when the discontinuities are confined to a finite collection of smooth submanifolds through which the flow passes transversally, as formalized in the following definitions.

\defn{\label{def:events}
Given a vector field $F:D\into TD$ over an open domain $D\subset\R^d$ and a function $h\in C^r(U,\R)$ defined on an open subset $U\subset D$,
we say that \emph{$h$ is an event function for $F$ on $U$} if there exists a positive constant $f > 0$ such that $Dh(x)F(x) \ge f$ for all $x\in U$.
A codimension--1 embedded submanifold $\Sigma\subset U$ for which $h|_\Sigma$ is constant is referred to as a \emph{{\localsec} for $F$}. 
}
\noindent
Note that if $h$ is an event function for $F$ on a set containing $\rho\in D$ then necessarily $Dh(\rho) \ne 0$.

We will show in \sct{flow} that vector fields that are differentiable everywhere except a finite collection of local sections give rise to a well--defined flow that is piecewise--differentiable.
This class of \emph{event--selected} vector fields is defined formally as follows.

\defn{\label{def:ecr}
Given a vector field $F:D\into TD$ over an open domain $D\subset\R^d$, we say that \emph{$F$ is {\ECr} at $\rho\in D$} if there exists an open set $U\subset D$ containing $\rho$ and a collection $\set{h_j}_{j=1}^n\subset C^r(U,\R)$ such that:
\begin{enumerate}
\item (event functions) $h_j$ is an event function for $F$ on $U$ for all $j\in\set{1,\dots,n}$;

\item ($C^r$ extension) for all $b\in\pmones = \cuben$, with 
  \eqn{D_b := \set{x\in U : b_j (h_j(x) - h_j(\rho)) \ge 0},} 
  $F|_{\Int{D_b}}$ admits a $C^r$ extension $F_b:U\into TU$.
\end{enumerate}

\noindent
(Note that for any $b\in\cuben$ such that $\Int{D_b} = \emptyset$ the latter condition is satisfied vacuously.)
We let $\ECrD$ denote the set of vector fields that are {\ECr} at every $x\in D$.
}

\noindent
For illustrations of {\ECr} vector fields in the plane $D = \R^2$, refer to Figures~\ref{fig:tau} and~\ref{fig:seq}.

\subsect{Construction of the Piecewise--Differentiable Flow}\label{sec:const}
The following constructions will be used to state and prove results throughout the chapter.
Suppose $F:D\into TD$ is {\ECr} at $\rho\in D$.
By definition there exists a neighborhood $\rho\in U\subset D$ and associated event functions $\set{h_j}_{j=1}^n\subset C^r(U,\R)$ that divide $U$ into regions $\set{D_b}_{b\in\cuben}$ by defined by $D_b := \set{x\in U : (h_j(x) - h_j(\rho)) b_j \ge 0}$.
The boundary of each $D_b$ is contained in the collection of event surfaces $\set{H_j}_{j=1}^n$ defined for each $j\in\set{1,\dots,n}$ by $H_j := \set{ x\in U :\, h_j(x)=h_j(\rho) }$.
For each $j\in\set{1,\dots,n}$ and $b\in\cuben$, we refer to the surface $H_j$ as an \emph{exit boundary in positive time} for $D_b$ if $h_j(D_b)\subset(-\infty,0]$; we refer to $H_j$ as an \emph{exit boundary in negative time} if $h_j(D_b)\subset[0,+\infty)$.
In addition, the definition of {\ECr} implies that there is a collection of $C^r$ vector fields $\set{F_b:U\into TU}_{b\in\cuben}\subset C^r(U,TU)$ such that $F|_{\Int{D_b}} = F_b|_{\Int{D_b}}$ for all $b\in\cuben$.

\subsubsection{Budgeted time--to--boundary}\label{sec:bttb}
For each $b\in\cuben$ with $\Int D_b\ne\emptyset$, let $\phi_b:\e{F}_b\into U$ be a flow for $F_b$ over a flow domain $\e{F}_b\subset\R\times U$ containing $(0,\rho)$; recall that $\phi_b\in C^r(\e{F}_b,U)$ since $F_b\in C^r(U,TU)$.
Each $H\in\set{H_j}_{j=1}^n$ is a {\localsec} for $F$, and therefore a {\localsec} for $F_b$ as well.
This implies $F_b(\rho)$ is transverse to $H$ (more precisely, $F_b(\rho)\not\in T_\rho H$), thus the Implicit Function Theorem~\cite[Theorem~C.40]{Lee2012} implies there exists a $C^r$ ``time--to--impact'' map $\tau_b^H:U_b^H\into \R$ defined on an open set $U_b^H\subset D$ containing $\rho$ such that
\eqnn{\label{eqn:taubH}
\forall x\in U_b^H : (\tau_b^H(x),x)\in\e{F}_b\ \text{and}\ \phi_b(\tau_b^H(x),x)\in H.
}
The collection of maps $\set{\tau_b^H}_{b\in\cuben}$ are jointly defined over the open set $U_b := \bigcap_{j=1}^n U_b^{H_j}$; note that $U_b$ is nonempty since $\rho\in U_b$.
Any $x\in U_b$ can be taken to any $H\in\set{H_j}_{j=1}^n$ by flowing with the vector field $F_b$ for time $\tau_b^H(x)\in\R$.
A useful fact we will recall in the sequel is that if $y = \phi_b(\tau_b(x),x)$ then 
\eqnn{\label{eqn:DtaubH}
D\tau_b^H(x) = \frac{-Dh(y)D_x\phi_b(t,x)}{Dh(y) F_b(y)}; 
}
this follows from~\cite[\S11.2]{HirschSmale1974}.

We now define functions $\tau_b^+,\tau_b^-:\R\times U_b\into\R$ that specify the time required to flow to the exit boundary of $D_b$ in forward or backward time, respectively, without exceeding a given time budget:
\eqnn{\label{eqn:taubp}
\forall (t,x)\in\R\times U_b : \tau_b^+(t,x) & = \max{\set{{0,\min\paren{\set{t}\cup\set{\tau_b^{H_j}(x) : b_j < 0}_{j=1}^n}}}}, \\
\forall (t,x)\in\R\times U_b : \tau_b^-(t,x) & = \min{\set{{0,\max\paren{\set{t}\cup\set{\tau_b^{H_j}(x) : b_j > 0}_{j=1}^n}}}}; 
}
Since $\tau_b^+,\tau_b^-$ are obtained via pointwise minimum and maximum of a finite collection of $C^r$ functions, we conclude $\tau_b^+,\tau_b^-\,\in PC^r(\R\times U_b,\R)$.
See \fig{tau} for an illustration of the component functions of $\tau_b^+$ in a planar vector field.
\begin{figure*}[t]
\centering
\includegraphics[width=13cm]{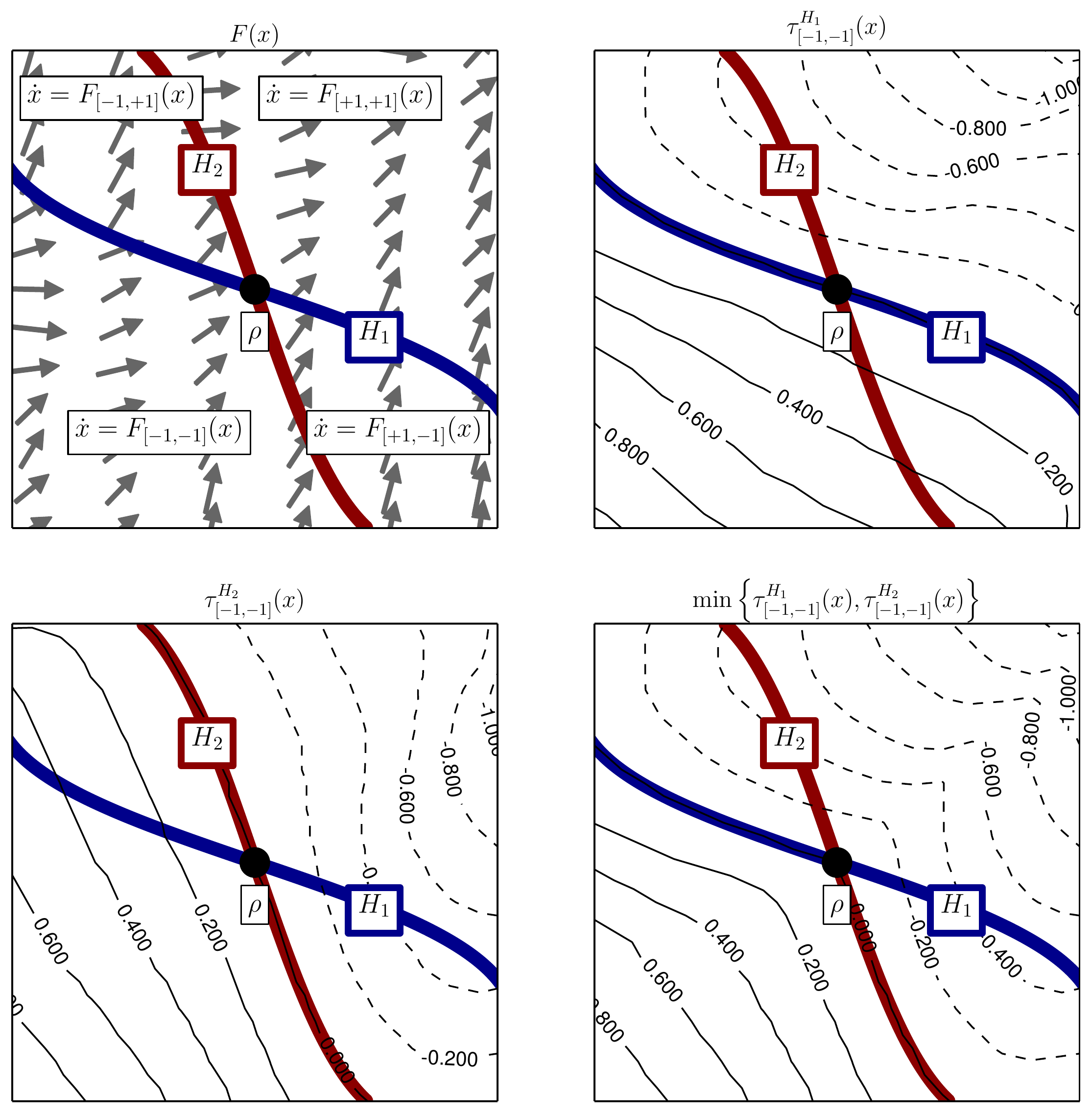}
\caption{
Illustration of a vector field $F:D\into TD$ that is {\ECr} at $\rho\in D = \R^2$.
The functions $\set{\tau_{[-1,-1]}^{H_j}}_{j=1}^2$ specify the time required to flow via the vector field $F_{[-1,-1]}$ to the surface $H_j$.
The pointwise minimum $\min\set{\tau_{[-1,-1]}^{H_j}(x)}_{j=1}^2$ is used in the definition of $\tau_{[-1,-1]}^+$ in~\eqref{eqn:taubp}.
}
\label{fig:tau}
\end{figure*}

In the sequel we will require the derivative of $\tau_b^+$ with respect to $t$ and $x$.
In general this can be obtained via the chain rule~\cite[Theorem~3.1.1]{Scholtes2012}.
If we define $\nu_b^+:U_b\into\R\cup\set{+\infty}$ using the convention $\min\emptyset = +\infty$ by 
\eqnn{\label{eqn:nub}
\forall x\in U_b : \nu_b^+(x) = \min\set{\tau_b^{H_j}(x) : b_j < 0}_{j=1}^n,
}
then we immediately conclude that for all $(t,x)\in\R\times U_b$ such that $\nu_b^+(x) \ne t \ne 0$,
the forward--time budgeted time--to--boundary 
$\tau_b^+$ is classically differentiable and
\eqnn{\label{eqn:Dtaubp}
D\tau_b^+(t,x) = \pw{(0,\ \zerod^\T),& t < 0; \\ (1,\ \zerod^\T),& 0 < t < \nu_b^+(x); \\ \paren{0, D\tau_b^H(x)},& \nu_b^+(x) < t;}
}
where in the third case $H\in\set{H_j}_{j=1}^n$ is such that $\tau_b^H(x) = \nu_b^+(x)$.
To compute $D\tau_b^-(t,x)$, one may simply use the formula in~\eqref{eqn:Dtaubp} applied to the vector field $-F$;
full details are provided in \app{bttb}.

\subsubsection{Budgeted flow--to--boundary}\label{sec:ftb}
By composing the flow $\phi_b$ with the budgeted time--to--boundary functions $\tau_b^+,\tau_b^-$, we now construct functions that flow points up to the exit boundary of $D_b$ in forward or backward time over domains
\eqn{
\e{V}_b^+ & = \set{(t,x)\in\R\times U_b : (\tau_b^+(t,x),x)\in\e{F}_b}, \\
\e{V}_b^- & = \set{(t,x)\in\R\times U_b : (\tau_b^-(t,x),x)\in\e{F}_b}.
}
(Note that $\e{V}_b^+,\e{V}_b^-$ are open since $\tau_b^+,\tau_b^-$ are continuous and nonempty since $(0,\rho)\in\e{V}_b^+,\e{V}_b^-$.)
For each $b\in\cuben$ define the functions $\zeta_b^+:\e{V}^+\into D,\zeta_b^-:\e{V}^-\into D$ by
\eqnn{\label{eqn:zetab}
\forall (t,x)\in\e{V}_b^+ : \zeta_b^+(t,x) & = \phi_b\paren{\tau_b^+(t,x),x}, \\
\forall (t,x)\in\e{V}_b^- : \zeta_b^-(t,x) & = \phi_b\paren{\tau_b^-(t,x),x}.
}
Clearly $\zeta_b^+\in PC^r(\e{V}_b^+, D)$ and $\zeta_b^-\in PC^r(\e{V}_b^-, D)$ since they are obtained by composing $PC^r$ functions~\cite[\S4.1]{Scholtes2012}.
Loosely speaking, the function $\zeta_b^+$ coincides with $\phi_b$ for pairs $(t,x)$ that do not cross the forward--time exit boundary of $D_b$.
Yet unlike $\phi_b$, it is the identity (stationary) flow over the remainder of its domain.
More precisely, for $t < 0$ and for values of 
$t > \nu_b^+(x)$
the function $\tau_b^+(t,x)$ is constant (and hence the derivative with respect to time $D_t \zeta_b^+(t,x) = 0$),
while for $t\in(0,\nu_b^+(x))$ we have $\zeta_b^+(t,x) = \phi_b(t,x)$ (and hence $D_t\zeta_b^+(t,x) = F_b(\phi_b(t,x))$).

Now fix $x\in D_b$, choose $a\in\cuben\sm{b}$, and for $t\in\R$ define 
\eqn{t_{a}^+(t) := \min\set{\tau_{a}^{H_j}\paren{\zeta_b^+(t,x)} : {a}_j < 0}_{j=1}^n.}
Applying the conclusions from the preceding paragraph, with $t'\in\R$ the composition 
\eqn{\zeta_{a}^+(t',\zeta_b^+(t,x))}
is classically differentiable with respect to both $t'$ and $t$ almost everywhere. 
Furthermore, we can deduce that the derivative of the composition with respect to $t$ is $F_{b}(\phi_b(t,x))$ when $t\in \paren{0,\nu_b^+(x)}$ and zero where it is otherwise defined;
similarly, the derivative with respect to $t'$ is $F_{a}\paren{\phi_{a}(t',\zeta_b^+(t,x)}$ when $t'\in\paren{0,t_{a}^+(t)}$ and zero where it is otherwise defined.
If we impose the relationship $t' := t - \tau_b^+(t,x)$, we have $t'=0$ for any $t\in (0,\nu_b^+(x))$.
The composition 
\eqn{\zeta_{a}^+(t - \tau_b^+(t,x),\zeta_b^+(t,x))}
follows the flow for $F_b$ from $x$ toward (but never passing) the exit boundary of $D_b$, then follows the flow of $F_{a}$ from $\zeta_b^+(t,x)$ toward the exit boundary of $D_{a}$.

In the sequel we will require the derivative of $\zeta_b^+$ with respect to $t$ and $x$.
In general this can be obtained via the chain rule~\cite[Theorem~3.1.1]{Scholtes2012}.
If we define $\nu_b^+:U_b\into\R$ as in~\eqref{eqn:nub} then we immediately conclude that for all $(t,x)\in\R\times U_b$ such that $\nu_b^+(x) \ne t \ne 0$, the forward--time flow--to--boundary $\zeta_b^+$ is classically differentiable and
\eqnn{\label{eqn:Dzetabp}
D\zeta_b^+(t,x) = \pw{(\zerod,\ \zerodxd),& t < 0; \\ (F_b(\phi_b(t,x)),\ D_x\phi_b(t,x)),& 0 < t < \nu_b^+(x); \\ (\zerod,\Upsilon(t,x)),& \nu_b^+(x) < t;}
}
where in the third case $\Upsilon(t,x) = {\paren{0, F_b(\phi_b(\tau_b^+(t,x),x)) D\tau_b^H(x) + D_x\phi_b(\tau_b^+(t,x),x)}}$ and $H\in\set{H_j}_{j=1}^n$ is such that $\tau_b^H(x) = \nu_b^+(x)$.
To compute $D\zeta_b^-(t,x)$, one may simply use the formula in~\eqref{eqn:Dzetabp} applied to the vector field $-F$;
full details are provided in \app{ftb}.

\subsubsection{Composite of budgeted time--to-- and flow--to--boundary}\label{sec:bttbftb}
Define $\vphi_b^+:\e{V}_b^+\into\R\times D$, $\vphi_b^-:\e{V}_b^-\into\R\times D$ by
\eqnn{\label{eqn:vphibp}
\forall (t,x)\in\e{V}_b^+ : \vphi_b^+(t,x) & = \paren{t - \tau_b^+(t,x), \zeta_b^+(t,x)} = \paren{t - \tau_b^+(t,x), \phi_b\paren{\tau_b^+(t,x),x}}, \\
\forall (t,x)\in\e{V}_b^- : \vphi_b^-(t,x) & = \paren{t - \tau_b^-(t,x), \zeta_b^-(t,x)} = \paren{t - \tau_b^-(t,x), \phi_b\paren{\tau_b^-(t,x),x}}.
}
Clearly $\vphi_b^+\in PC^r(\e{V}_b^+,\R\times D)$ and $\vphi_b^-\in PC^r(\e{V}_b^-,\R\times D)$.
Intuitively, the second component of the $\vphi_b^+$, $\vphi_b^-$ functions flow according to $F_b$ up to exit boundaries of $D_b$ in forward or backward time, respectively, while the first component deducts the flow time $t - \tau_b^\pm(t,x)$ from the total time budget $t$.
These functions satisfy an invariance property:
\eqnn{\label{eqn:vphi:inv}
\forall (t,x)\in\paren{\e{V}_b^+\cap (-\infty,0]\times U_b} &: \vphi_b^+(t,x) = (t,x), \\
\forall (t,x)\in\paren{\e{V}_b^-\cap [0,+\infty)\times U_b} &: \vphi_b^-(t,x) = (t,x). 
}

We now combine~\eqref{eqn:Dtaubp} and~\eqref{eqn:Dzetabp} to obtain the derivative of $\vphi_b^+$ for all $(t,x)\in\R\times U_b$ such that $\nu_b^+(x) \ne t \ne 0$:
\eqnn{\label{eqn:Dvphibp}
D\vphi_b^+(t,x) = \pw{\mat{cc}{1 & \zerod^\T \\ \zerod & I_d},& t < 0; \\ \mat{cc}{0 & \zerod^\T \\ F_b(\phi_b(t,x)) & D_x\phi_b(t,x)},& 0 < t < \nu_b^+(x); \\ \mat{cc}{1 & -D\tau_b^H(x) \\ \zerod & \Upsilon(t,x)},& \nu_b^+(x) < t;}
}
where in the third case $\Upsilon(t,x) = {F_b(\phi_b(\tau_b^+(t,x),x)) D\tau_b^H(x) + D_x\phi_b(\tau_b^+(t,x),x)}$ and $H\in\set{H_j}_{j=1}^n$ is such that $\tau_b^H(x) = \nu_b^+(x)$.
To compute $D\vphi_b^-(t,x)$, one may simply use the formula in~\eqref{eqn:Dzetabp} applied to the vector field $-F$;
full details are provided in \app{bttbftb}.

\subsubsection{Construction of flow via composition}
Consider now the formal composition
\eqnn{\label{eqn:phi}
\phi = \pi_2\circ\paren{\prod_{b=-\ones}^{+\ones} \vphi_{b}^+}\circ\paren{\prod_{b=+\ones}^{-\ones} \vphi_{b}^-}
}
where $\pi_2:\R\times D\into D$ is the canonical projection and $\prod_{b=-\ones}^{+\ones}$ denotes composition in lexicographic order (similarly $\prod_{b=+\ones}^{-\ones}$ denotes composition in reverse lexicographic order).
The set $\phi^{-1}(D)\subset\R\times D$ is open (since $\phi$ is continuous) and nonempty (since combining~\eqref{eqn:vphi:inv} and~\eqref{eqn:phi} implies $\phi(0,\rho) = \rho$).
Therefore there exist open neighborhoods $J\subset\R$ of $0$ and $V\subset D$ of $\rho$ such that $\e{F} = J\times V\subset\phi^{-1}(D)$.
Clearly $\phi\in PC^r(\e{F},D)$ since it is obtained by composing $PC^r$ functions.
Its derivative can be computed by applying the chain rule~\cite[Theorem~3.1.1]{Scholtes2012};
alternatively, it can be obtained for almost all $(t,x)\in\e{F}$ as a product of the appropriate matrices given in~\eqref{eqn:Dvphibp},~\eqref{eqn:Dvphibm}.
The derivative with respect to time has a particularly simple form almost everywhere, as we demonstrate in the following Lemma.

\begin{figure*}[t]
\centering
\includegraphics[width=13cm]{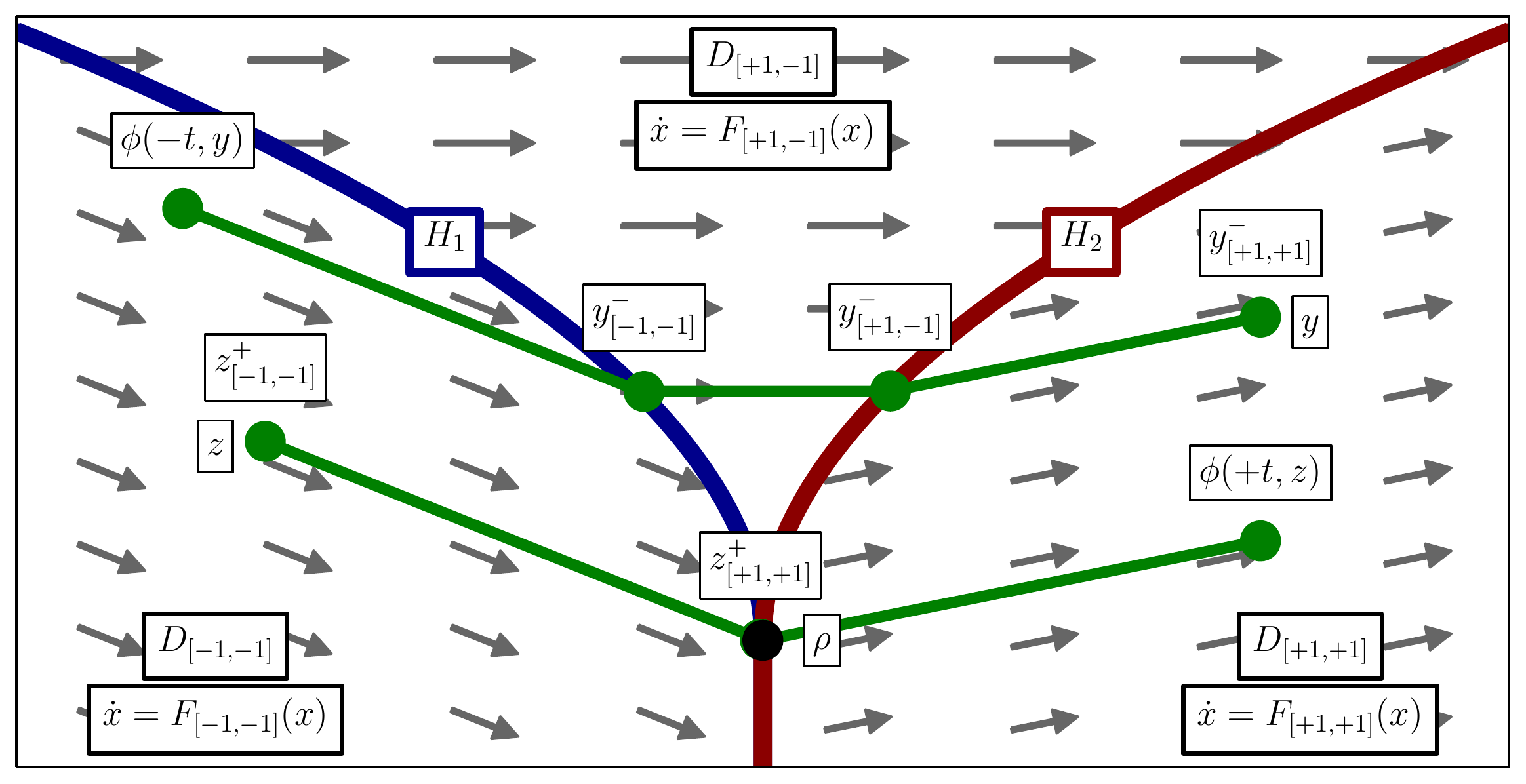}
\caption{\label{fig:seq}
Illustration of a vector field $F:D\into TD$ that is {\ECr} near $\rho\in D = \R^2$.
The vector field is discontinuous across the $C^r$ codimension--1 submanifolds $H_1,H_2\subset D$.
For each $b\in\cuben = \set{[-1,-1], [+1,-1], [-1,+1], [+1,+1]}$, if $\Int D_b\ne\emptyset$ then the vector field restricts as $F|_{\Int D_b} = F_b|_{\Int D_b}$ where $F_b:U_b\into TU_b$ is a smooth vector field over a neighborhood $\rho\in U_b\subset D$.
An initial condition $z\in D_{[-1,-1]}$ flows in forward time to $\phi(+t,y)\in D_{[+1,+1]}$ through $z_{[+1,+1]}^+\in H_1\cap H_2$.
An initial condition $y\in D_{[+1,+1]}$ flows in backward time to $\phi(-t,y)\in D_{[-1,-1]}$ through $y_{[-1,+1]}^-\in H_1$ and $y_{[-1,-1]}^-\in H_2$.
}
\end{figure*}

\begin{lemma}[time derivative of flow]
\label{lem:Dflow}
If the vector field $F:D\into TD$ is {\ECr} at $\rho\in D$, 
then
for almost all $(t,x)\in\e{F}$ 
the flow
$\phi\in PC^r(\e{F},D)$ defined by~\eqref{eqn:phi} is differentiable with respect to time and
\eqnn{\label{eqn:der}
D_t\phi(t,x) = F(\phi(t,x)).
}
\end{lemma}

\pf{
Choose $x\in D$ such that $(0,x)\in\e{F}$.  We will show that $\phi|_{\e{F}^x}$ is classically differentiable for almost all times $t\in\e{F}^x$.
Let $t^- = \inf\e{F}^x$, $t^+ = \sup\e{F}^x$ so that $0 \in \e{F}^x = (t^-,t^+)$.
We construct a partition of $[0,t^+)$ as follows.
For each $b\in\cuben$, let $(t_b^+,x_b^+) = \paren{\prod_{a < b} \vphi_{a}^+}(t^+,x)$ where the composition is over all $a\in\cuben$ that occur before $b$ lexicographically; refer to \fig{seq} for an illustration of the sequence $\set{y_b}_{b\in\cuben}$ generated by an initial condition $y\in D_{-\ones}$.
Note that $\set{t^+ - t_b^+}_{b\in\cuben}$ is (lexicographically) non--decreasing and $t_{+\ones}^+ = \tau_{+\ones}^+(t_{+\ones}^+,x_{+\ones}^+)$.
Defining the interval
\eqn{
J_b = [t^+ - t_b^+, t^+ - t_b^+ + \tau_b^+(t_b^+,x_b^+)],
}
we have $[0,t^+) \subset\bigcup_{b\in\cuben} J_b^+$ and $\Int J_a^+\cap\Int J_b^+ = \emptyset$ for all $a\in\cuben\sm\set{b}$.
Observe that 
\eqn{
\forall t\in\Int J_b^+ : \phi(t,x) = \pi_2\circ\vphi_b^+(t - (t^+-t_b^+),x_b^+) \in \Int D_b,
}
where the condition is vacuously satisfied if $\Int J_b^+ = \emptyset$.
Therefore for all $t\in\Int J_b^+$, the piecewise--differentiable function $\phi$ is classically differentiable with respect to time at $(t,x)$ and we have
\eqn{
D_t \phi(t,x) & = D\pi_2\, D_t \vphi_b^+(t - (t^+-t_b^+),x_b^+) \\
& = F_b(\pi_2\circ\vphi_b^+(t - (t^+-t_b^+),x_b^+)) \\
& = F(\pi_2\circ\vphi_b^+(t - (t^+-t_b^+),x_b^+)) \\ 
& = F(\phi(t,x)).
}
Applying an analogous argument in backward time, we conclude that $D_t \phi(t,x) = F(\phi(t,x))$ for almost all $t\in (t^-,t^+) = \e{F}^x$.
Since $(0,x)\in \e{F}$ was arbitrary, the Lemma follows.
}

\subsect{Piecewise--Differentiable Flow}\label{sec:flow}
We now show that the piecewise--differ-entiable function $\phi\in PC^r(\e{F},D)$ defined in~\eqref{eqn:phi} is in fact a flow for the discontinuous vector field $F$.
See \fig{seq} for an illustration of this flow.

\begin{theorem}[local flow]
\label{thm:flow}
Suppose 
the vector field $F:D\into TD$ 
is {\ECr} at $\rho\in D$.
Then there exists a flow $\phi:\e{F}\into D$ for $F$ over a flow domain $\e{F}\subset\R\times D$ containing $(0,\rho)$ such that $\phi\in PC^r(\e{F},D)$ and 
\eqnn{\label{eqn:flow}
\forall (t,x)\in\e{F} : \phi(t,x) = x + \int_0^t F(\phi(s,x))\, ds.
}
\end{theorem}

\pf{
We claim that $\phi\in PC^r(\e{F},D)$ from~\eqref{eqn:phi} satisfies~\eqref{eqn:flow}.  
Applying the fundamental theorem of calculus~\cite[Proposition~3.1.1]{Scholtes2012} in conjunction with Lemma~\ref{lem:Dflow} and positive--homogeneity of the derivative~\eqref{eqn:der}, we find
\eqn{
\phi(t,x) & = \phi(0,x) + \int_0^1 D\phi(tu,x;t,0) du \\ 
& = x + \int_0^t D\phi(s,x;t,0) \frac{1}{t} ds \\
& = x + \int_0^t D_t \phi(s,x) ds \\
& = x + \int_0^t F(\phi(s,x)) ds.
}
}

If the vector field $F:D\into TD$ is {\ECr} at every point in the domain $D$, we may stitch together the local flows obtained from {\thmflow} to obtain a global flow.

\begin{corollary}[global flow]
\label{cor:flow}
If $F\in \ECrD$, then there exists a unique maximal flow $\phi\in PC^r(\e{F},D)$ for $F$.
This flow has the following properties:
\begin{enumerate}
\item[(a)] For each $x\in D$, the curve $\phi^x:\e{F}^x\into D$ is the unique maximal integral curve of $F$ starting at $x$.
\item[(b)] If $s\in\e{F}^x$, then $\e{F}^{\phi(s,x)} = \e{F}^x - s = \set{t - s : t\in\e{F}^x}$.
\item[(c)] For each $t\in\R$, the set $D_t = \set{x\in D : (t,x)\in\e{F}}$ is open in $D$ and $\phi_t:D_t\into D_{-t}$ is a piecewise--$C^r$ homeomorphism with inverse $\phi_{-t}$.
\end{enumerate}
\end{corollary}

\pf{
This follows from a straightforward modification of the analogous Theorem~9.12 in~\cite{Lee2012} (simply replace all occurrences of the word ``smooth'' with ``$PC^r$'').
We recapitulate the argument in 
\iftoggle{siads}
{
\app{flow} in the Supplemental Materials.
}
{
\app{flow}.
}
}

If a vector field is {\ECr} at every point along an integral curve, the following Lemma shows that it is actually $C^r$ at all but a finite number of points along the curve.

\begin{lemma}[$EC^r$ implies $C^r$ almost everywhere]
\label{lem:finite}
Suppose 
the vector field $F:D\into TD$ 
is {\ECr} at every point along an integral curve $\xi:I\into D$ for $F$ over a compact interval $I\subset\R$.
Then there exists a finite subset $\delta\subset\xi(I)$ such that $F$ is $C^r$ on $\xi(I)\sm\delta$.
\end{lemma}

\pf{
Let $\delta\subset\xi(I)$ be the set of points where $F$ fails to be $C^r$.
If $\abs{\delta} = \infty$, then since $\xi(I)$ is compact 
there exists an accumulation point $\alpha\in\xi(I)$. 
Since $F$ is {\ECr} at $\alpha$, there exists $\veps > 0$ such that $F$ is $C^r$ at every point in the set $(B_\veps(\alpha)\cap\xi(I))\sm\set{\alpha}$, but this violates the existence of an accumulation point $\alpha\in\delta$.
Therefore $\abs{\delta} < \infty$.
}

\rem{
One of the major values of {\thmflow} lies in the fact that piecewise--differentiable functions possess a first--order approximation called the \emph{Bouligand derivative} as described in \sct{pcrnsa}.
This Bouligand derivative (or B--derivative) is weaker than the classical (Fr\'{e}chet) derivative, but significantly stronger%
\footnote{For instance, the B--derivative enables constructions like the Fundamental Theorem of Calculus~\cite[Proposition~3.1.1]{Scholtes2012} and Inverse Function Theorem~\cite[Corollary~20]{RalphScholtes1997} not enjoyed by the directional derivative.}
than the directional derivative.
The B--derivative of the composition~\eqref{eqn:phi} can be computed by applying the chain rule~\cite[Theorem 3.1.1]{Scholtes2012}.
}

\sect{Time--to--Impact, (Poincar\'{e}) Impact Map, and Flowbox}\label{sec:impact}

We now leverage the fact that {\ECr} vector fields yield piecewise--differentiable flows to obtain useful constructions familiar from classical (smooth) dynamical systems theory.
Using an inverse function theorem~\cite[Corollary~20]{RalphScholtes1997}, we construct time--to--impact maps for local sections in \sct{tti}.
We then apply this construction to infer the existence of piecewise--differentiable ({\Poincare}) impact maps associated with periodic orbits in \sct{Pmap} and piecewise--differentiable flowboxes in \sct{fbox}.

\subsection{Piecewise--Differentiable Time--to--Impact}\label{sec:tti} 

We begin in this section by constructing piecewise--differentiable time--to--impact maps.

\begin{theorem}[time--to--impact]
\label{thm:impact}
Suppose 
the vector field $F:D\into TD$ 
is {\ECr} at $\rho\in D$.
If $\sigma\in C^r(U,\R)$ is an event function for $F$ on an open neighborhood $U\subset D$ of $\rho$, 
then there exists an open neighborhood $V\subset D$ of $\rho$ and piecewise--differentiable function $\mu\in PC^r(V,\R)$ such that
\eqnn{\label{eqn:time}
\forall x\in V : \sigma\circ\phi(\mu(x),x) = \sigma(\rho)
}
where $\phi\in PC^r(\e{F},D)$ is a flow for $F$ and $(0,\rho)\in\e{F}$.
\end{theorem}

\begin{proof}
{\thmflow} ensures the existence of a flow $\phi\in PC^r(\e{F},D)$ such that $\e{F}\subset\R\times D$ contains $(0,\rho)$.
Let $\alpha = \sigma\circ\phi$, and note that there exist open neighborhoods $T\subset\R$ of $0$ and $W\subset D$ of $\rho$ such that $\alpha\in PC^r(T\times W,\R)$. 

We aim to apply an implicit function theorem to show that $\alpha(s,x) = \sigma(\rho)$ has a unique piecewise--differentiable solution $s = \mu(x)$ near $(0,\rho)$.
To do so, we need to establish the function $\alpha$ is \emph{completely coherently oriented} with respect to its first argument.

Specializing Definition~16 in~\cite{RalphScholtes1997}, a sufficient condition for $\alpha$ to be completely coherently oriented with respect to its first argument at $(0,\rho)$ is that the (scalar) derivatives $D \alpha_j(0,\rho;1,0)$ of all essentially active selection functions $\set{\alpha_j : j\in I^e(\alpha,(0,\rho))}$ have the same sign.
Lemma~\ref{lem:Dflow} implies the time derivatives of all essentially active selection functions for $\phi$ at $(0,\rho)$ are contained in the collection $\set{F_b(\rho) : b\in\cuben, D_b\neq\emptyset}$
where $\set{F_b : b\in\cuben}$ are the $C^r$ vector fields that define $F$ near $\rho$.
Since $\sigma$ is an event function for $F$, there exists $f > 0$ such that
\eqn{
\forall b\in\cuben : D\sigma(\rho) F_b(\rho) \ge f > 0.
}
This implies $\alpha$ is completely coherently oriented with respect to time at $(0,\rho)$.
Therefore we may apply Corollary~20 in~\cite{RalphScholtes1997} to obtain an open neighborhood $0\in V\subset\R$ and a piecewise--differentiable function $\mu\in PC^r(V,\R)$ such that~\eqref{eqn:time} holds.
\end{proof}

\begin{corollary}[time--to--impact]
\label{cor:impact}
Suppose 
the vector field $F:D\into TD$ 
is {\ECr} at every point along an integral curve $\xi:[0,t]\into D$ for $F$.
If $\sigma\in C^r(U,\R)$ is an event function for $F$ on an open set $U\subset D$ containing $\xi(t)$, 
then there exists an open neighborhood $V\subset D$ of $\xi(0)$ and piecewise--differentiable function $\mu\in PC^r(V,\R)$ that satisfies~\eqref{eqn:time}.
\end{corollary}

\begin{proof}
Corollary~\ref{cor:flow} ensures the existence of a flow $\phi\in PC^r(\e{F},D)$ such that $\e{F}\subset\R\times D$ contains $[0,t]\times\set{\xi(0)}$.
Let $\td{\mu}\in PC^r(\td{V},\R)$ be the time--to--impact for $\sigma$ obtained by applying Corollary~\ref{thm:impact} at $\xi(t) = \phi(t,\xi(0))$.
Then with $V = \set{x\in D : \phi(t,x)\in\td{V}}$, noting that $V$ is nonempty since $\xi(0)\in V$ and open since $\phi$ is continuous, the function $\mu:V\into\R$ defined by $\mu(x) = t + \td{\mu}\circ\phi(t,x)$ is piecewise--$C^r$ and satisfies~\eqref{eqn:time}.
\end{proof}

\subsect{Piecewise--Differentiable (Poincar\'{e}) Impact Map}\label{sec:Pmap}

We now apply {\thmimpact} in the important case where the integral curve is a \emph{periodic orbit} to construct a piecewise--differentiable (Poincar\'{e}) impact map.

\defn{\label{def:gamma}
An integral curve $\gamma:\R\into D$ is a \emph{periodic orbit} for the vector field $F:D\into TD$ if there exists $t > 0$ such that $\gamma(t) = \gamma(0)$ and $D_t\gamma(s) \ne 0$ for all $s\in [0,t]$.
The minimal $t > 0$ for which $\gamma(t) = \gamma(0)$ is referred to as the \emph{period} of $\gamma$, and we say that $\gamma$ is a $t$--periodic orbit for $F$.
We let $\Gamma = \gamma(\R)$ denote the image of $\gamma$.
}

Suppose the vector field $F:D\into TD$ 
is {\ECr} at every point along a $t$--periodic orbit $\gamma$ for $F$.
Then given a {\localsec} $\Sigma\subset D$ for $F$ that intersects $\Gamma = \gamma(\R)$ at $\set{\fp} = \Gamma\cap\Sigma$,
Corollary~\ref{cor:impact} implies there exists a piecewise--differentiable time--to--impact $\mu\in PC^r(V,\R)$ defined over an open neighborhood $V\subset D$ of $\fp$ such that $\mu(\fp) = t$.
With $V\cap\Sigma$, we let $\ipct:V\into\Sigma$ be the piecewise--differentiable impact map defined by
\eqnn{\label{eqn:ipct}
\forall x\in V : \ipct(x) = \phi(\mu(x),x).
}

\begin{theorem}[Poincar\'{e} map]
\label{thm:pmap}
Suppose the vector field $F:D\into TD$ 
is {\ECr} at every point along a periodic orbit $\gamma$ for $F$.
Then given a {\localsec} $\Sigma\subset D$ for $F$ that intersects $\Gamma = \gamma(\R)$ at $\set{\fp} = \Gamma\cap\Sigma$,
there exists an open neighborhood $V\subset D$ of $\fp$ such that 
the impact map~\eqref{eqn:ipct} restricts to a piecewise--differentiable (Poincar\'{e}) map $P\in PC^r(S,\Sigma)$ on $S = V\cap\Sigma$.
\end{theorem}

\pf{
Without loss of generality assume $\gamma(0)\in\Sigma$.
Let $T$ be the period of $\gamma$, apply {\thmimpact} to $\gamma|_{[0,T]}$ to obtain an open set $V\subset D$ containing $\gamma(0)$ and a piecewise--$C^r$ impact time map $\mu\in PC^r(V,\R)$, and define $\ipct:V\into\Sigma$ as in~\eqref{eqn:ipct}.
Then with $S = V\cap\Sigma$, the restriction $P = \ipct|_S$ is a piecewise--$C^r$ Poincar\'{e} map for $\gamma$.
}

Since the Poincar\'{e} map $P:S\into\Sigma$ yielded by {\thmpmap} is piecewise--differentiable, it admits a first--order approximation (its Bouligand derivative) $DP:TS\into T\Sigma$ that can be used to assess local exponential stability of the fixed point $P(\fp) = \fp$.
This topic will be investigated in more detail in \sct{stab}.

\subsect{Piecewise--Differentiable Flowbox}\label{sec:fbox}

{\thmimpact} enables us to easily derive a canonical form for the flow near an event--selected vector field discontinuity.
\begin{theorem}[flowbox]
\label{thm:flowbox}
Suppose the vector field $F:D\into TD$ is {\ECr} at $\rho\in D$, and let $\phi:\e{F}\into D$ be the flow obtained from {\thmflow}.
Then there exists a piecewise--differentiable homeomorphism $\flbx\in PC^r(V,W)$ between neighborhoods $V\subset D$ of $\rho$ and $W\subset\R^d$ of $0$ such that 
\eqn{\forall x\in V, t\in\e{F}^x : \flbx\circ\phi(t,x) = \flbx(x) + t e_1}
where $e_1\in\R^d$ is the first standard Euclidean basis vector.
\end{theorem}

\pf{
Let $\sigma\in C^r(U,\R)$ be an event function for $F$ on a neighborhood $\rho\in U\subset D$ that is linear\footnote{Existence of a linear event function is always guaranteed.  For instance, take the linear approximation at $\rho$ of any nonlinear event function for $F$}.
{\thmimpact} implies there exists a piecewise--differentiable time--to--impact map $\mu\in PC^r(V,\R)$ on a neighborhood $V\subset D$ of $\rho$ such that
\eqn{
\forall x\in V : \sigma\circ\phi(\mu(x),x) = \sigma(\rho),
}
i.e. $\phi(\mu(x),x)$ lies in the codimension--1 subspace $\Sigma = \sigma^{-1}(\sigma(\rho))$.
Define $\flbx:V\into\R\times\Sigma$ by 
\eqnn{\label{eqn:psi}
\forall x\in V : \flbx(x) = (-\mu(x), \phi(\mu(x),x)).
}
Clearly $\flbx\in PC^r(V,\R\times\Sigma)$ and hence $\flbx$ is continuous.
Furthermore, it is clear that $\flbx$ is injective since (i) $\pi_\Sigma\flbx(x) = \pi_\Sigma\flbx(y)$ implies $x$ and $y$ lie along the same integral curve, and (ii) distinct points along an integral curve pass through $\Sigma$ at distinct times. 
It follows from Brouwer's Open Mapping Theorem~\cite{Brouwer1911, Hatcher2002} that the image $W = \flbx(V)$ is an open subset of $\R^d$.
This implies $\flbx$ is a homeomorphism between $V$ and $W$.
With $\iota:\R\times\Sigma\into\R\times D$ denoting the canonical inclusion,
the inverse of $\flbx\in PC^r(V,W)$ is $\phi\circ\iota|_W\in PC^r(W,V)$,
thus $\flbx$ is a $PC^r$ homeomorphism.
Finally, using the semi--group property of the flow $\phi$ and the fact that $\mu\circ\phi(t,x) = \mu(x) - t$ for all $x\in V, t\in\e{F}^x$,
\eqn{
\forall x\in V,\ t\in\e{F}^x : \flbx\circ\phi(t,x) & = \paren{-\mu\circ\phi(t,x), \phi(\mu\circ\phi(t,x),\phi(t,x))} \\ 
& = \paren{t - \mu(x), \phi(\mu(x) - t, \phi(t,x))} \\
& = \paren{t - \mu(x), \phi(\mu(x),x)} \\
& = \flbx(x) + t e_1.
}
}

\noindent
Thus the flow is conjugate via a piecewise--differentiable homeomorphism to a \emph{flowbox}~\cite[\S11.2]{HirschSmale1974},~\cite[Theorem~9.22]{Lee2012}.

\sect{Perturbed Flow}\label{sec:pert}
In this section we study how the flow associated with an {\ECr} vector field varies under perturbations to both the smooth vector field components (in \sct{pertvf}) and the event functions (in \sct{pertevt}).

\subsection{Perturbation of Vector Fields}\label{sec:pertvf}
Suppose $F:D\into TD$ is {\EC{r}} at $\rho\in D$ with respect to the components of $h\in C^r(D,\R^n)$.
Then by Definition~\ref{def:ecr} there exists $U\subset D$ containing $\rho$ such that for each $b\in\cuben$ 
either $\Int{D_b} = \emptyset$ or $D_b\subset U$ and $F|_{\Int{D_b}}$ admits a $C^r$ extension $F_b:U\into TU$.
We note that $F$ is determined on $U$ up to a set of measure zero from $h$ and the function $\ha{F}\in\Ccoprod{U}$ defined by $\ha{F}|_{\set{b}\times U} = F_b|_U$.
Note that we regard $\Ccoprod{U}$ as a vector space under pointwise addition of tangent vectors and the norm
\eqnn{\label{eqn:Ccoprod}
\norma{C^r}{\ha{F}} = \sum_{b\in\cuben}\norma{C^r}{\ha{F}|_{\set{b}\times U}}.
}
Thus in the sequel we consider perturbations to {\ECr} vector fields in the space $\Ccoprod{U}$.

\begin{theorem}[vector field perturbation]
\label{thm:ss1}
Let $F\in\Ccoprod{D}$, $h\in C^r(D,\R^n)$ determine an {\EC{r}} vector field at $\rho\in D$, $r \ge 1$.
Then for all $\veps > 0$ there exists $\delta > 0$ such that
for all $\td{F}\in \BC{r}{\delta}(F)$: 
\begin{enumerate}[label=(\alph*)]
\item pairing $h$ with the perturbed vector field $\td{F}$ determines an {\EC{r}} vector field at $\rho$;
\item the perturbed flow $\td{\phi}:\td{\e{F}}\into D$ obtained by applying {\thmflow} to this perturbed vector field satisfies $\td{\phi}\in\BC{0}{\veps}(\phi)$ on $\td{\e{F}}\cap\e{F}$ and $(0,\rho)\in \td{\e{F}}\cap\e{F}$;
\item there exists a piecewise--differentiable homeomorphism $\eta\in PC^r(U,\td{U})$ defined between neighborhoods $U,\td{U}\subset D$ of $\rho$ such that 
$\eta|_{B_\delta(\rho)}\in\BC{0}{\veps}(\id_{B_\delta(\rho)})$ and
we have
\eqnn{
\eta\circ\phi(t,x) = \td{\phi}(t,\eta(x))
}
for all $(t,x) \in \R\times\R^d$ such that $x\in U$, $t\in \e{F}^x\cap \td{\e{F}}^{\eta(x)}$, and $\phi(t,x)\in U$.
\end{enumerate}
\end{theorem}

\begin{proof}
Since $F$ is {\ECr} with respect to $h$ at $\rho$, there exists $f > 0$ such that for all $x$ sufficiently close to $\rho$ every component of $Dh(x) F(x)$ is bounded below by $f$.
Then so long as $0 < \delta < f$, every component of $Dh(x) \td{F}(x)$ is bounded below by $f - \delta > 0$, establishing claim (a). 

We claim that (b) follows from~\cite[Theorem~1 in \S8 of Chapter~2]{Filippov1988}, which we reproduce as {\thminc} in 
\iftoggle{siads}
{
\app{filippov} in the Supplemental Materials.
}
{
\app{filippov}.
}
Indeed, given any $G\in C^r(\coprod_{b\in\cuben} D, \coprod_{b\in\cuben} TD)$ for which $(G,h)$ determines an {\ECr} vector field, define a set--valued map $\obar{G}:D\into \pset{TD}$ as follows:
\eqnn{\label{eqn:setval}
\forall x\in D : \obar{G}(x) = \conv\set{G|_{\set{b}\times D}(x) : b\in\cuben,\ x\in D_b}.
}
At any $x\in D$, it is clear that $\obar{G}(x)$ is nonempty, bounded, closed, and convex.
Furthermore, it is clear that $\obar{G}$ is upper semicontinuous at $x$ in the sense defined in \sct{pcrnsa}.
Therefore the map $\obar{G}$ satisfies {\assinc} over the domain of the flow for $G$.
It is straightforward to verify that solutions to the differential inclusion
$\dot{x} \in \obar{G}(x)$ coincide with those of the differential equation $\dot{x} = G(x)$ since the derivatives of the (absolutely continuous) solution functions agree almost everywhere.
Claim (b) then follows by applying {\thminc} to $\obar{F}$ determined from $F$ by~\eqref{eqn:setval} and $\obar{\td{F}}$ determined from $\td{F}\in\BC{r}{\delta}(F)$ by~\eqref{eqn:setval}.

For claim (c), apply {\thmflowbox} to $\phi$ and $\td{\phi}$ to obtain $\flbx\in PC^r(V,W)$ and $\td{\flbx}\in PC^r(\td{V},\td{W})$ such that
\eqnn{
\forall x\in V\cap\td{V}, t\in\e{F}^x\cap \td{\e{F}}^x : \flbx\circ\phi(t,x) = \flbx(x) + t\,e_1,\ \td{\flbx}\circ\td{\phi}(t,x) = \td{\flbx}(x) + t\,e_1.
}
Then with $U = \flbx^{-1}(\td{W})$, $\td{U} = \td{\flbx}^{-1}\circ\flbx(U)$ (both sets are nonempty since $\rho\in U\cap\td{U}$ and open since $\flbx$ and $\td{\flbx}$ are homeomorphisms), the piecewise--differentiable homeomorphism $\eta = \td{\flbx}^{-1}\circ\flbx|_{U}\in PC^r(U,\td{U})$ provides conjugacy between $\phi$ and $\td{\phi}$
for all $(t,x) \in \R\times\R^d$ such that $x\in U$, $t\in \e{F}^x\cap \td{\e{F}}^{\eta(x)}$, and $\phi(t,x)\in U$:
\eqnn{
\eta\circ\phi(t,x) & = \td{\flbx}^{-1}\circ\flbx\circ\phi(t,x) \\ 
& = \td{\flbx}^{-1}\paren{\flbx(x) + t\,e_1} \\
& = \td{\flbx}^{-1}\paren{\td{\flbx}\circ\td{\flbx}^{-1}\circ\flbx(x) + t\,e_1} \\
& = \td{\flbx}^{-1}\paren{\td{\flbx}\circ\eta(x) + t\,e_1} \\
& = \td{\phi}(t,\eta(x)).
}
We now wish to choose $\delta > 0$ sufficiently small to ensure 
$\eta|_{B_\delta(\rho)}\in\BC{0}{\veps}(\id_{B_\delta(\rho)})$.
Recalling from~\eqref{eqn:psi},
\eqnn{
\forall x\in V : \flbx(x) = (-\mu(x), \phi(\mu(x),x)),
}
where $\mu\in PC^r(V,\R)$ is the time--to--impact map for the event surface used to define $\flbx$,
we have
\eqnn{\label{eqn:psitd}
\norm{\flbx(x) - \td{\flbx}(x)} 
& \le \abs{\mu(x) - \td{\mu}(x)} + \norm{\phi(\mu(x),x) - \td{\phi}(\td{\mu}(x),x)} \\
& \le \abs{\mu(x) - \td{\mu}(x)} + \norm{\phi(\mu(x),x) - \phi(\td{\mu}(x),x)} \\
& \quad + \norm{\phi(\td{\mu}(x),x) - \td{\phi}(\td{\mu}(x),x)} \\
& \le (1 + L_\phi)\abs{\mu(x) - \td{\mu}(x)} + \veps_\phi
}
where $L_\phi > 0$ is a Lipschitz constant for $\phi$ on $\obar{B_\delta(0,\rho)}$, claim (b) ensures $\td{\phi}\in\BC{0}{\veps_\phi}(\phi)$ for any desired $\veps_\phi > 0$, and we have restricted to $x\in V\cap\td{V}\cap B_\delta(\rho)$ for which $(\td{\mu}(x),x)\in\e{F}$ and $(\mu(x),x)\in\td{\e{F}}$.
Applying~\cite[Lemma~9, Theorem~11]{RalphScholtes1997} to $\mu$, we conclude that $\delta > 0$ can be chosen sufficiently small to ensure $\td{\mu}\in\BC{0}{\veps_\mu}(\mu)$ for any desired $\veps_\mu > 0$.
Therefore $(1 + L_\phi)\veps_\mu + \veps_\phi$ can be made arbitrarily small in~\eqref{eqn:psitd}, hence we may apply~\cite[Theorem~11]{RalphScholtes1997} to choose $\delta > 0$ sufficiently small to ensure $\td{\flbx}^{-1}\in\BC{0}{\veps}(\flbx^{-1})$ for any desired $\veps > 0$.
Thus $\delta > 0$ may be chosen sufficiently small to ensure $B_\delta(\rho)\subset U$ and
\eqnn{
\norm{\eta(x) - x} 
& = \norm{\td{\flbx}^{-1}\circ\flbx(x) - \flbx^{-1}\circ\flbx(x)} \le \veps,
}
whence $\eta|_{B_\delta(\rho)}\in\BC{0}{\veps}(\id_{B_\delta(\rho)})$.
This completes the proof of claim (c).
\end{proof}

\subsect{Perturbation of Event Functions}\label{sec:pertevt}
It is a well--known fact that the solution of $n$ equations in $n$ unknowns generically varies continuously with variations in the equations.
This observation provides a basis for studying structural stability of the flow associated with {\ECr} vector fields when there are exactly $n = d = \dim D$ event functions, since for a collection of event functions $\set{h_j}_{j=1}^d\subset C^r(D,\R)$ whose composite $h\in C^r(D,\R^d)$ satisfies $\det Dh(\rho)\ne 0$, the existence of a unique intersection point $\td{\rho}$ and the set of possible transition sequences undertaken by nearby trajectories are unaffected by a sufficiently small perturbation $\td{h}$ of $h$.
We now combine this observation with the previous Theorem.
Subsequently, we will present an embedding technique that enables immediate generalization to cases where $Dh(\rho)$ is not invertible (whether because $n < d$, $n > d$, or $n = d$ and $\det Dh(\rho) = 0$).

\begin{theorem}[event function perturbation]
\label{thm:ss2}
Let $F\in\Ccoprod{D}$, $h\in C^r(D,\R^d)$ determine an {\EC{r}} vector field at $\rho\in D$ and suppose $Dh(\rho)$ is invertible, $r \ge 1$.
Then for all $\veps > 0$ sufficiently small there exists $\delta > 0$ such that
for all $\td{F}\in \BC{r}{\delta}(F)$, $\td{h}\in \BC{r}{\delta}(h)$: 
\begin{enumerate}[label=(\alph*)]
\item there exists a unique $\td{\rho}\in B_\delta(\rho)$ such that $\td{h}(\td{\rho}) = 0$ and $\td{h}(x)\ne 0$ for all $x\in B_\delta(\rho)\sm\set{\td{\rho}}$;
\item pairing $\td{h}$ with the perturbed vector field $\td{F}$ determines an {\EC{r}} vector field at $\td{\rho}$;
\item the perturbed flow yielded by {\thmflow}, $\td{\phi}:\td{\e{F}}\into D$, satisfies $\td{\phi}\in\BC{0}{\veps}(\phi)$ on $\td{\e{F}}\cap\e{F} \ne \emptyset$;
\item there exists a piecewise--differentiable homeomorphism $\eta\in PC^r(U,\td{U})$ defined between neighborhoods $U,\td{U}\subset D$ containing $\set{\rho,\td{\rho}}$ such that 
$\eta|_{B_\delta(\rho)}\in\BC{0}{\veps}(\id_{B_\delta(\rho)})$ and
we have
\eqnn{
\eta\circ\phi(t,x) = \td{\phi}(t,\eta(x))
}
for all $(t,x) \in \R\times\R^d$ such that $x\in U$, $t\in \e{F}^x\cap \td{\e{F}}^{\eta(x)}$, and $\phi(t,x)\in U$.
\end{enumerate}
\end{theorem}

%
%

\begin{proof}[Proof of {\thmpertef}]
Smooth dependence of the intersection point follows from the Implicit Function Theorem~\cite[Theorem~2.5.7]{AbrahamMarsden1988} since $C^r$ functions over compact domains comprise a Banach space~\cite[Chapter~2.1]{Hirsch1976}.
Specifically, if $h\in C^r(D,\R^n)$ satisfies $h(\rho) = 0$ for some $\rho\in D$ and $Dh(\rho)$ is invertible\footnote{Note that necessarily $n = \dim D$.},
then there exists $\alpha,\beta > 0$ and $\td{\rho}\in C^r(B_\alpha(h),B_\beta(\rho))$ such that
for all $\td{h}\in B_\alpha(h)$ the point $\td{\rho}(\td{h})$ is the unique zero of $\td{h}$ on $B_\beta(\rho)$, i.e. $\td{h}(\td{\rho}(\td{h})) = 0$ and for all $x\in B_\beta(\rho)\sm\set{\td{\rho}(\td{h})}$ we have $\td{h}(x) \ne 0$.
This establishes (a); (b) follows from continuity.

For any $\delta' > 0$, we can choose $\delta > 0$ sufficiently small to ensure that $\td{F}\in\BC{r}{\delta}(F)$, $\td{h}\in\BC{r}{\delta}(h)$ implies $D\td{h}^{-1}\circ\td{F}\in\BC{r}{\delta'}\paren{Dh^{-1}\circ F}$; 
let $\td{F}' = D\td{h}^{-1}\circ\td{F}$, $F' = Dh^{-1}\circ F$.
With $\td{\phi}':\td{\e{F}}'\into\R^d$, ${\phi}':{\e{F}}'\into\R^d$ denoting the flows for $\td{F}'$, $F'$,
{\thmpertvf} implies that $\delta'>0$ can be chosen sufficiently small to ensure $\td{\phi}'\in\BC{0}{\veps'}(\phi')$ for any $\veps' > 0$.
Since $\td{h}$ provides conjugacy between $\td{\phi}$ and $\td{\phi}'$, and similarly $h$ provides conjugacy between $\phi$ and $\phi'$,
we conclude that $\delta > 0$ can be chosen sufficiently small to ensure $\td{\phi}\in\BC{0}{\veps}(\phi)$ on $\td{\e{F}}\cap\e{F}$.
This establishes (c).

Let $\eta'\in PC^r(U',\td{U}')$ be the conjugacy from {\thmpertvf} relating $\phi'$ to $\td{\phi}'$.
Then $\eta = \td{h}\circ\eta'\circ h^{-1}$ provides conjugacy between $\phi$ and $\td{\phi}$ since
\eqnn{
\eta\circ\phi(t,x) 
& = \td{h}\circ\eta'\circ h^{-1}\circ\phi(t,x) \\
& = \td{h}\circ\eta'\circ\phi'(t,h^{-1}(x)) \\
& = \td{h}\circ\td{\phi}'(t,\eta'\circ h^{-1}(x)) \\
& = \td{\phi}(t,\td{h}\circ\eta'\circ h^{-1}(x)) \\
& = \td{\phi}(t,\eta(x))
}
for all $(t,x) \in \R\times\R^d$ such that $x\in h(U)$, $t\in \e{F}^x\cap \td{\e{F}}^{\eta(x)}$, and $\phi(t,x)\in U$.
Furthermore, given $\veps > 0$ we may choose $\delta > 0$ sufficiently small to ensure 
$\td{h}^{-1}\in\BC{0}{\delta}(h^{-1})$ and $\eta'\in\BC{0}{\delta}(\id)$, 
whence
\eqnn{
\norm{\eta(x) - x}
& = \norm{\td{h}\circ\eta'\circ h^{-1}(x) - x} \\
& \le \norm{\td{h}\circ\eta'\circ h^{-1}(x) - \td{h}\circ h^{-1}(x)} + \norm{\td{h}\circ h^{-1}(x) - x} \\
& \le L_{\td{h}}\norm{\eta'(y) - y} + \delta \\
& \le (1 + L_{\td{h}})\delta
}
for all $x\in B_\delta(0)$.
Thus $\delta < \veps / (1 + L_{\td{h}})$ ensures $\eta|_{B_\delta(\rho)}\in\BC{0}{\veps}(\id_{B_\delta(\rho)})$.
This completes the proof of claim (d).
\end{proof}


\rem{
Now consider the case where $F:D\into TD$ is {\ECr} at $\rho\in D$ with respect to the composite event function $h\in C^r(D,\R^n)$ but $Dh(\rho)\in\R^{n\times d}$ is not invertible (because either $n < d$, $n > d$, or $n = d$ and $\det Dh(\rho) = 0$).
We will embed this $d$--dimensional system into a $(d+n)$--dimensional system to obtain an {\ECr} vector field with respect to an invertible composite event function; this will enable application of the preceding Theorem to the degenerate system determined by $F$ and $h$.
For each $b\in\cuben$, let $S_b = \set{x^\T\in\R^{1\times d} : x^\T F_b(\rho) > 0}$ be the open half--space of row vectors that have a positive inner product with $F_b(\rho)$.
The set $S = \cap_{b\in\cuben} S_b$ is open (since each $S_b$ is open) and nonempty (since in particular $Dh_1(\rho)\in S$). 
Let $A\in\R^{d\times d}$ be an invertible matrix whose rows are selected from $S$; 
such a matrix always exists since $S$ is open and nonempty.
Now let $\obar{D} = D\times\R^n$ and define $\obar{F}:\obar{D}\into T\obar{D}$ and $\obar{h}\in C^r(\obar{D},\R^{d+n})$ as follows:
\eqnn{
\forall (x,y)\in D\times\R^n : \obar{F}(x,y) = \mat{c}{F(x) \\ 0},\ \obar{h}(x,y) = \mat{c}{A x \\ h(x) + y}.
}
Clearly $\obar{F}$ is {\ECr} at $\obar{\rho} = (\rho, 0)$,
and $D\obar{h}(\obar{\rho})$ is invertible since
\eqnn{
D\obar{h}(\obar{\rho}) = \mat{cc}{A & 0 \\ Dh(x) & I_n}\in\R^{(d+n)\times(d+n)}
}
has linearly independent columns.
Therefore {\thmpertef} may be applied to study the effect of perturbations on the flow $\obar{\phi}:\obar{\e{F}}\into\obar{D}$ for $\obar{F}$;
the conclusions of the Theorem can be specialized to the original flow $\phi:\e{F}\into D$ for $F$ as follows.
With $\e{V} = \set{(t,x)\in\e{F} : (t,x,0)\in\obar{\e{F}}}$ let
$\iota:\e{V}\into \td{\e{F}}$ denote the embedding defined by $\iota(t,x) = (t,x,0)$ for all $(t,x)\in\e{V}$ and let 
$\pi:\obar{D}\into D$ denote the projection defined by $\pi(x,y) = x$ for all $(x,y)\in\obar{D}$.
With these definitions we have
\eqnn{
\phi|_{\e{V}} = 
\paren{
\pi\circ
\obar{\phi}\circ
\iota
}|_{\e{V}}.
}
}

\sect{Computation}\label{sec:comp}
In this section, we apply the theoretical results from Sections~\ref{sec:floww},~\ref{sec:impact}, and~\ref{sec:pert} to derive procedures to compute the B--derivative of the flow and assess stability of a periodic orbit for an {\ECr} vector field $F$.
We begin in 
\sct{salt} by developing a concrete procedure to compute the B--derivative of the piecewise--differentiable flow yielded by $F$.
Subsequently, in \sct{stab} we provide sufficient conditions ensuring exponential stability of a periodic orbit that passes through the intersection of multiple surfaces of discontinuity for $F$.

\subsect{Variational Equations and Saltation Matrices}\label{sec:salt}

In this section we compute the B--derivative of the piecewise--differentiable flow by solving a jump--linear time--varying ordinary differential equation (ODE) along a trajectory.
At trajectory points where the vector field is $C^r$, we recall in \sct{Cr} that the derivative is obtained by solving a time--varying ODE (the so--called \emph{variational equation}) with no ``jumps''.
At points where the vector field is discontinuous along one (or two transverse) event surface(s), in \sct{ECr} we note (as others have before us) that the ODE must be updated discontinuously (via a so--called \emph{saltation matrix}).
In the remainder of the section, we exploit properties of the piecewise--differentiable flow to derive a generalization of this procedure applicable in the presence of an arbitrary number of surfaces of discontinuity that are not required to be transverse.

\subsubsect{$C^r$ vector field}\label{sec:Cr}
Let $D\subset\R^d$ be an open domain and $F\in C^r(D,TD)$ a smooth vector field on $D$.
It is a classical result~\cite[Theorem~1 in~\S15.2]{HirschSmale1974} that the derivative of the flow $\phi:\e{F}\into D$ associated with $F$ with respect to state can be obtained by solving a linear time--varying  differential equation---the so--called \emph{variational equation}---along a trajectory, i.e. if  $(t,x)\in\e{F}$ and $X:[0,t]\into\R^{d\times d}$ satisfies 
\eqnn{\label{eqn:var}
\forall u\in[0,t] : \dot{X}(u) = D_x F(\phi(u,x)) X(u),\ X(0) = I,
}
then the derivative of the flow with respect to time and state is given by
\eqnn{\label{eqn:DtDxphi}
D_t \phi(t,x) = F(\phi(t,x)),\ D_x \phi(t,x) = X(t).
}
Here and in the sequel we assume without loss of generality that $t > 0$; the $t < 0$ case can be addressed by applying the same reasoning to the vector field $-F$.

\subsubsect{Event--selected $C^r$ vector field}\label{sec:ECr}
If the vector field is instead {\ECr}, $F\in EC^r(D)$, adjustments must be made to~\eqref{eqn:var} wherever a trajectory crosses a surface of discontinuity.
Let $\phi:\e{F}\into D$ denote the global flow of $F$ yielded by~{\corflow} and let $(t,x)\in\e{F}$.
As shown in~\cite[Equation~1.4]{AizermanGantmacher1958} (and subsequently~\cite[Equations~57--60]{HiskensPai2000}), if for some $s\in(0,t)$ the vector field $F$ is {\ECr} at $\rho = \phi(s,x)$ with respect to a single surface of discontinuity, $H$, then the variational equation~\eqref{eqn:var} must be updated discontinuously via multiplication by a so--called \emph{saltation matrix},
\eqnn{\label{eqn:salt:Ivanov}
X(s^+) = \brak{I + \frac{\paren{F_{+\ones}(\rho) - F_{-\ones}(\rho)} Dh(\rho)}{Dh(\rho) F_{-\ones}(\rho)}} X(s^-),
}
where $X(s^+) = \lim_{u\goesto s^+} X(u)$,
$X(s^-) = \lim_{u\goesto s^-} X(u)$,
and $H\subset h^{-1}(0)$ near $\rho$.

As claimed in~\cite[Equation~2.4]{Ivanov1998} (and subsequently~\cite[Theorem~7.5]{Di-BernardoBudd2008}, ~\cite[Equation~46]{DieciLopez2011}, and~\cite[Equation~27]{BizzarriBrambilla2013}), if for some $s\in(0,t)$ the
vector field $F$ is {\ECr} at $\rho = \phi(s,x)$ with respect to multiple surfaces of discontinuity, then the variational equation~\eqref{eqn:var} must be updated discontinuously via multiplication by one saltation matrix for each surface.
Unlike the preceding cases, the flow will generally not possess a classical derivative with respect to state after time $s$.
Previous authors compute the first--order effect of the flow using crossing times of perturbed trajectories.
Due to the combinatorial complexity of this approach, these authors only derive the first--order approximation for two intersecting surfaces; though they claim that the approach readily extends to arbitrary numbers of intersecting surfaces, they leave the details to the reader.

The development in \sct{floww} enables us to directly compute the derivative of the flow along trajectories passing through an arbitrary collection $\set{H_j}_{j=1}^n$ of surfaces across which $F$ is discontinuous.
Without loss of generality\footnote{Lemma~\ref{lem:finite} ensures there are a finite number of discontinuities along any integral curve of a vector field $F\in EC^r(D)$.  Therefore to evaluate the B--derivative of the flow after any number of discontinuities one may iteratively apply the procedure described in the sequel to a finite number of trajectory segments and combine the result using the chain rule~\cite[Theorem~3.1.1]{Scholtes2012}.} we assume $F$ is $C^r$ at every point in $\phi([0,t]\sm \set{s}, x)$, and we let $\rho = \phi(s,x)$ as before.

\subsubsect{Sampled vector field associated with {\ECr} vector field}
We begin by noting that the B--derivative calculation in~\eqref{eqn:Dvphibp} depends only on first--order approximations of the flow and event functions $\set{h_j}_{j=1}^n$.
For all $b\in\cuben$ let 
\eqnn{
\td{D}_b = \set{x\in D : b_j\, Dh_j(\rho)(x - \rho) \ge 0}
} 
and consider the flow $\td{\phi}:\td{\e{F}}\into D$ of the piecewise--constant vector field $\td{F}:D\into TD$ defined by
\eqnn{\label{eqn:Fsamp}
\forall b\in\cuben,\ x\in \td{D}_b : \td{F}(x) = F_b(\rho).
}
Applying~\eqref{eqn:Dvphibp} together with the chain rule~\cite[Theorem~3.1.1]{Scholtes2012} we conclude that
\eqnn{\label{eqn:Dphisamp}
\forall (v,w)\in T_{(0,\rho)}\e{F} : D\phi(0,\rho; v,w) = D\td{\phi}(0,\rho; v,w).
}
In other words, by sampling the {\ECr} vector field $F$ across its tangent planes we obtain a piecewise--constant {\EC{\infty}} vector field $\td{F}$ whose flow $\td{\phi}$ agrees with the flow $\phi$ for $F$ to first order.
In this sense, we regard the piecewise--constant ``sampled'' vector field $\td{F}$ as the analogue of the linearization of a smooth vector field in our nonsmooth setting.
Note that, since the flow of the sampled system is obtained in~\eqref{eqn:phi} by composing a sequence of piecewise--affine functions, it is piecewise--affine:
\eqnn{\label{eqn:Dphiaffine}
\forall (v,w)\in T_{(0,\rho)}\e{F} : \td{\phi}(v,\rho + w) = \td{\phi}(0, \rho) + D\td{\phi}(0,\rho; v,w).
}
These observations enable us in the remainder of this section to derive several properties of the B--derivative that will prove useful in the applications presented in \sct{app}.

\subsubsect{Saltation matrix for multiple transition surfaces}
Suppose $(v,w)\in T_{(t,x)}\e{F} = \R\times\R^d$
is such that\footnote{Since the flow $\td{\phi}$ for the ``sampled'' vector field~\eqref{eqn:Fsamp} is piecewise--affine, the set of tangent vectors that fail to satisfy the two specified conditions has measure zero.  Since the B--derivative is a continuous function of tangent vectors, it is determined by its values on the dense subset of tangent vectors that satisfy the condition.}
for all $c > 0$ sufficiently small the trajectory initialized at
$x + cw$:
(i) passes through a unique sequence of $m$ region interiors on its way to $\phi(t + cv, x + cw)\in D_{+\ones}$;
and
(ii) does not pass through the intersection of non--tangent surfaces.
Let $\word:\set{1,\dots,m}\into \cuben$ specify the sequence of region interiors, excluding $D_{+\ones}$, 
and let 
$\eta:\set{1,\dots,m}\into\set{1,\dots,n}$ 
specify the corresponding sequence%
\footnote{If $H_j$ is tangent to $H_i$ at $\rho$ then either $H_j$ or $H_i$ may be indexed by $\eta$; the choice will have no effect on the subsequent calculation.} 
of surfaces crossed. 
The B--derivative of the flow evaluated in the $(v,w)$ direction is
\eqnn{\label{eqn:Dphi}
D\phi(t,x; v,w) = D\phi(t-s, \rho)\brak{\prod_{j=1}^m D\vphi_{\word(j)}^+(0,\rho)}\mat{c}{0 \\ D\phi(s,x)} \mat{c}{v \\ w},
}
where
$D\phi(t-s,\rho)$, $D\phi(s,x)$ are obtained as in~\eqref{eqn:DtDxphi} by solving the classical variational equation since $F$ is smoothly extendable to a neighborhood of those segments of the trajectory
and for each $j\in\set{1,\dots,m}$ the derivative $D\vphi_{\word(j)}^+(0,\rho)$ is given by the matrix in the third case in~\eqref{eqn:Dvphibp} with the simplifications $\tau_{\word(j)}^+(0,\rho) = 0$, $\phi_{\word(j)}(0,\rho) = \rho$.
Substituting $f = F_{\word(j)}(\rho)$, $g^\T = D h_{\eta(j)}(\rho)$ for clarity yields
\eqnn{\label{eqn:Dvphibps}
D\vphi_{\word(j)}^+(0,\rho) = \mat{cc}{1 & \frac{1}{g^\T f} g^\T \\ 0 & I - \frac{1}{g^\T f}f\,g^\T} =
I + \frac{1}{g^\T f} \mat{c}{1 \\ -f} \mat{cc}{0 & g^\T}
}
since~\eqref{eqn:DtaubH} simplifies to $D\tau_{\word(j)}^{H_{\eta(j)}}(\rho) = -\frac{1}{g^\T f}g^\T$.
Thus, the $B$--derivative in~\eqref{eqn:Dphi} is obtained by composing rank--1 updates of the identity with solutions to classical variational equations.
In the sequel we will make use of the saltation matrix $\Xi_\word\in\R^{(d+1)\times(d+1)}$ given by
\eqnn{\label{eqn:salt}
\Xi_\word = \prod_{j=1}^m D\vphi_{\word(j)}^+(0,\rho).
}

\subsubsect{Flow between tangent transition surfaces}
If the surfaces are tangent at the point $\rho = \phi(s,x)\in\bigcap_{j=1}^n H_j \ne \emptyset$ of intersection with the trajectory, 
a perturbed trajectory is not affected to first order by flow through the interior of a region between surfaces that are tangent;
this follows from the equality in~\eqref{eqn:Dphisamp} relating the B--derivative of the original system to that of its ``sampled'' version. 
Indeed, consider the vector field illustrated in \fig{seq} where the surfaces $H_1$ and $H_2$ are tangent at $\rho$.
Evaluating the derivative $D\phi(t,z;0,(0,\delta))$ for any $\delta > 0$ requires composition of $D\vphi_{-\ones}$, $D\vphi_{[+1,-1]}$, and $D\vphi_{+\ones}$,
\eqn{
D\phi(t,z;0,(0,\delta)) = D\phi(t-s,\rho) D\vphi_{[+1,-1]}^+(0,\rho) D\vphi_{+\ones}^+(0,\rho) D\phi(s,z) \mat{c}{0 \\ \mat{c}{0 \\ \delta}},
}
since the perturbed trajectory $\phi(t,z + (0,\delta))$ passes through the interior of $D_{[+1,-1]}$.
Combining~\eqref{eqn:DtaubH},~\eqref{eqn:Dphi}, and~\eqref{eqn:Dvphibps}, 
after some algebra we obtain
\eqn{
D\vphi_{[+1,-1]}^+(0,\rho) D\vphi_{-\ones}^+(0,\rho) = D\vphi_{-\ones}^+(0,\rho).
}
In other words, $D\phi(t,z;0,(0,\delta))$ is unaffected by flow through $D_{[+1,-1]}$.
Intuitively, the time spent flowing through any region between surfaces that meet at a tangency at $\rho\in D$ depends quadratically on the distance from $\rho$, therefore it does not affect the first--order approximation of the flow through $\rho$.
If $r > 1$ B--derivatives of the flow are desired, then it would be necessary to take these higher--order effects into account when evaluating the desired higher--order derivative.

\subsubsect{Variational equation for event--selected $C^r$ vector field}
By synthesizing the preceding observations, we now provide a generalization of the variational equation in~\eqref{eqn:var} applicable to the piecewise--differentiable flow yielded by an {\ECr} vector field.
We wish to evaluate $D\phi(t,x;v,w)$ where $F$ is {\ECr} at $\rho = \phi(s,x)$ for some $s\in(0,t)$ and $F$ is $C^r$ at every point in $\phi([0,t]\sm\set{s}, x)$, and where $(v,w)\in T_{(t,x)} \e{F}$.
By~\eqref{eqn:Dphi}, the desired derivative can be obtained by solving a jump--linear time--varying differential equation.
With $\word:\set{1,\dots,m}\into\cuben$ denoting the \emph{word} associated with the tangent vector $(v,w)$ from~\eqref{eqn:Dphi} and letting $\Xi_\word\in\R^{(d+1)\times(d+1)}$ be the saltation matrix from~\eqref{eqn:salt},
if $(\lambda,\xi):[0,t]\into\R\times\R^d$ satisfies
\eqnn{\label{eqn:varn}
\forall u\in [0,t]\sm\set{s} :
\mat{c}{\dot{\lambda}(u) \\ \dot{\xi}(u)} &= \mat{c}{0 \\ D_x F(\phi(u,x)) \xi(u)},\\
\mat{c}{\lambda(0) \\ \xi(0)} = \mat{c}{v \\ w}&,\
\mat{c}{\lambda(s) \\ \xi(s)} = \Xi_\word \mat{c}{\lambda(s^-) \\ \xi(s^-)},
}
the B--derivative of the flow is given by
\eqnn{
D\phi(t,x; v,w) = F(\phi(t,x))\lambda(t) + \xi(t).
}
More generally,~\eqref{eqn:Dphi} indicates the selection functions for the piecewise--differentiable flow $\phi$ are indexed by the set of \emph{words}, i.e. functions from $\set{1,\dots,m}$ into $\cuben$ that specify the sequence of regions a perturbed trajectory could pass through when flowing from $D_{-\ones}$ to $D_{+\ones}$:
\eqnn{\label{eqn:words}
\Words = \set{\word:\set{1,\dots,m}\into\cuben \st m\le n,\ \word\ \text{is injective and increases from $-\ones$ to $+\ones$}};
}
here the phrase \emph{$\word$ increases from $-\ones$ to $+\ones$} means that $\word(1) = -\ones$, $\word(m) = +\ones$, and for each $j\in\set{1,\dots,m-1}$ there exists $I_j\subset\set{1,\dots,n}$ such that $\word(j+1) - \word(j) = 2\sum_{i\in I_j}e_i$.
To evaluate the (Fr\'{e}chet) derivative for the selection function $\phi_\word$ indexed by $\word\in\Words$, we solve a matrix--valued jump--linear time--varying differential equation to obtain $(\Lambda^\T_\word,X_\word):[0,t]\into\R^{(d+1)\times d}$ via
\eqnn{\label{eqn:salt:mat}
\forall u\in [0,t]\sm\set{s} :
\mat{c}{\dot{\Lambda}^\T_\word(u) \\ \dot{X}_\word(u)} &= \mat{c}{0 \\ D_x F(\phi(u,x)) X_\word(u)},\\
\mat{c}{\Lambda^\T_\word(0) \\ X_\word(0)} = \mat{c}{0 \\ I_d}&,\
\mat{c}{\Lambda^\T_\word(s) \\ X_\word(s)} = \Xi_\word \mat{c}{\Lambda^\T_\word(s^-) \\ X_\word(s^-)}.
}
Then the B--derivative of the selection function $\phi_\word$ with respect to state is given by
\eqnn{\label{eqn:Dxphiword}
D_x\phi_\word(t,x) = F(\phi(t,x)) \Lambda^\T_\word(t) + X_\word(t)
}
As we demonstrate in the following section,
evaluating~\eqref{eqn:Dxphiword} for all words $\word\in\Words$ provides a straightforward computational procedure%
\footnote{Though straightforward, this procedure can be laborious since the number of elements in $\Words$ grows factorially with the number $n$ of surfaces of discontinuity.} 
to check contractivity of a Poincar\'{e} map associated with a periodic orbit.

\subsect{Stability of a Periodic Orbit}\label{sec:stab}
We assume given an {\ECr} vector field $F\in EC^r(D)$ over an open domain $D\subset\R^d$ containing a periodic orbit $\gamma:\R\into D$.
{\thmflow} and {\corflow} together yield a maximal flow $\phi\in PC^r(\e{F},D)$ for $F$.
{\thmpmap} yields a Poincar\'{e} map $P\in PC^r(S,\Sigma)$ over any local section $\Sigma\subset D$ that intersects $\Gamma = \gamma(\R)$ at $\set{\fp} = \Gamma\cap\Sigma$.
The Bouligand derivative $DP:TS\into T\Sigma$ of this piecewise--differentiable Poincar\'{e} map can be used to assess local exponential stability of the fixed point $P(\fp) = \fp$,
as the following Corollary shows; this generalizes Proposition~3 in~\cite{Ivanov1998} to stability of fixed points for arbitrary $PC^r$ functions.

\begin{proposition}[contractivity test for stability of a periodic orbit]
\label{prop:stab}
Suppose $P\in PC^r(S,\Sigma)$ where $S\subset\Sigma$ has a fixed point $P(\fp) = \fp$ and $DP$ is a contraction over tangent vectors near $\fp$, i.e.
there exists $c\in (0,1)$, $\delta > 0$, and $\norm{\cdot}:\R^{d-1}\times\R^{d-1}\into\R$ such that 
\eqnn{\label{eqn:DPlec}
\forall x\in B_\delta(\fp)\subset S\cap\Sigma,\ v\in T_x\Sigma : \norm{DP(x;v)} \le c\norm{v}.
}
Then $\gamma$ is an exponentially stable periodic orbit.
\end{proposition}

\pf{
By the fundamental theorem of calculus~\cite[Proposition~3.1.1]{Scholtes2012}, 
\eqn{
\forall x,y\in \obar{B_\delta(\fp)} : \norm{P(x) - P(y)} & \le \int_0^1 \norm{DP(y + s(x - y); x - y)} ds \\
& \le c \norm{x - y}.
}
We conclude that $P$ is a contraction over the compact ball $\obar{B_\delta(\fp)}$, whence by the {\CMT} its unique fixed point $P(\fp) = \fp$ is exponentially stable.
}

\noindent
In the remainder of this section we consider the case where $P$ is a {\Pmap} associated with a periodic orbit in an {\ECr} vector field, and demonstrate how the B--derivative of $P$ can be obtained from the B--derivative of the flow $\phi$. 
This provides a straightforward computational procedure to determine whether the contraction hypothesis in the above Proposition is satisfied using the variational equation developed in \sct{salt}.

Let $\mu\in C^r(V,\R)$ be the time--to--impact map for $\Sigma$ on a neighborhood $V\subset D$ containing $\fp$; note that $V$ can be chosen sufficiently small to ensure $\mu$ is continuously (as opposed to piecewise) differentiable since $F$ is $C^r$ at $\fp$.
Let $\ipct\in C^r(V,\Sigma)$ be the impact map given by $\ipct(x) = \phi(\mu(x),x)$ for all $x\in V$; again note that $\ipct$ is continuously differentiable.
By continuity of the flow there exists a neighborhood $U\subset S\subset\Sigma$ of $\fp$ sufficiently small to ensure $\set{\phi(t,x) : x\in U}\subset V$, whence we have the equality
\eqnn{
\forall x\in U : P(x) = \ipct\circ\phi(t,x).
}
Applying the~{\chainrule} we find that
\eqnn{\label{eqn:DpsiDphi}
\forall w\in T_\fp\Sigma : DP(\fp; w) = D\ipct(\fp) D\phi(t,\fp; 0, w)
}
where $D\ipct(\fp)\in\R^{(d-1)\times d}$ is the ({\Frechet}) derivative of $\ipct$. 
Following the conventions from \sct{salt}, let $\set{\phi_\word}_{\word\in\Words}$ denote the set of selection functions for the flow $\phi$.
Now satisfying the contractivity condition~\eqref{eqn:DPlec} from~{\propstab}, namely that $DP$ is a contraction over tangent vectors near $\fp$, is clearly equivalent to finding $c\in (0,1)$ and $\norm{\cdot}:\R^{d-1}\times\R^{d-1}\into\R$ such that
\eqnn{\label{eqn:DpsiDphilec}
\forall\word\in\Words,\ w\in T_\fp\Sigma : \norm{D\ipct(\fp) D_x\phi_\word(t,\fp)(0, w)} \le c\norm{w}.
}
We emphasize that a single norm must be found relative to which the inequality in~\eqref{eqn:DpsiDphilec} is satisfied for all $\word\in\Words$; it would not suffice, for instance, to merely ensure that all the eigenvalues of $D\ipct(\fp) D_x\phi_\word(t,\fp)$ reside in the open unit ball.

The condition in~\eqref{eqn:DpsiDphilec} is equivalent to requiring that the \emph{induced norm} of the linear operator $D\ipct(\fp) D_x\phi_\word(t,\fp)$ satisfy the bound
\eqnn{\label{eqn:DpsiDphi:inorm}
\forall\word\in\Words : \inorm{D\ipct(\fp) D_x\phi_\word(t,\fp)} \le c.
}
These observations are summarized formally in the following Proposition.

\begin{proposition}[induced norm test for periodic orbit stability]\label{prop:inorm}
Let $D$ be an open domain, suppose $\gamma:\R\into D$ is a $t$--periodic orbit for $F\in EC^r(D)$, let $\phi\in PC^r(\e{F},D)$ denote the maximal flow for $F$, and let $\set{\phi_\word}_{\word\in\Words}$ denote a set of selection functions for $\phi$.
Let $\Sigma\subset D$ be a local section for $F$ such that $F$ is $C^r$ at $\set{\fp} = \Gamma\cap\Sigma$ where $\Gamma = \gamma(\R)$ and let $\ipct\in C^r(V,\R)$ be the impact map for $\Sigma$ over a neighborhood $V\subset D$ containing $\fp$ such that $F|_V$ is $C^r$.
If there exists $c \in (0,1)$ and $\norm{\cdot}:\R^d\into\R^d$ such that~\eqref{eqn:DpsiDphi:inorm} holds,
then $\gamma$ is an exponentially stable periodic orbit.
\end{proposition}

\rem{
As noted above,~\eqref{eqn:DpsiDphi:inorm} is equivalent to stipulating that $DP$ is a contraction over tangent vectors near $\fp$, which is the contractivity condition from~{\propstab}.
In~\cite{Ivanov1998}, Ivanov considered the stability of a fixed point of a piecewise--defined map.
It is clear from his exposition that~\cite[Proposition~3]{Ivanov1998} is intended to apply to the Poincar\'{e} map $P$ associated with a periodic orbit that passes through multiple surfaces of discontinuity. 
We demonstrate that $P$ has the piecewise--defined form assumed in~\cite[(3.1)]{Ivanov1998} and formally derive a stability condition in~{\propinorm} that is equivalent to that in~\cite[Proposition~3]{Ivanov1998}.
}

\rem{
In~{\propinorm}, the problem of finding the norm that ensures~\eqref{eqn:DpsiDphi:inorm} holds is equivalent to that of finding a common quadratic Lyapunov function for a switched linear system, which remains an open problem in the theory of switched systems.
We refer the interested reader to~\cite[Section~II--A]{LinAntsaklis2009} for a survey of state--of--the--art approaches to this problem.
}

\sect{Applications}\label{sec:app}
We now illustrate the applicability of these results by appeal to a very simple family of {\ECr} fields that abstractly captures the essential nature of the discontinuities arising in the physical settings mentioned in \sct{intro}.
For instance, integrate--and--fire neuron models consist of a population of $n$ subsystems that undergo a discontinuous change in membrane voltage and synaptic capacitance triggered by crossing a voltage threshold~\cite{KeenerHoppensteadt1981, HopfieldHerz1995, BizzarriBrambilla2013}.
Since the discontinuities in state are confined to independent ``reset'' translations in membrane voltages~\cite[Equation~(2)]{KeenerHoppensteadt1981}, these transitions can be modeled locally as a first--order discontinuity in an {\ECr} vector field.
As another example, legged animals and robots with four, six, and more limbs exhibit gaits with near--simultaneous touchdown of two or more legs~\cite{Alexander1984, GolubitskyStewart1999, HolmesFull2006}. 
Since each touchdown introduces a discontinuity in the forces acting on the body, these transitions give rise to second--order discontinuities in an {\ECr} vector field.
In the context of electrical power networks, when constituent elements---lines, cables, and transformers---encounter excessive voltages or currents they trip fail--safe mechanisms that discontinuously change connectivity between elements~\cite[Section~II-A.2]{Hiskens1995}.  

Motivated by these applications in neuroscience, biological and robotic locomotion, and electrical engineering, we now apply the results derived in the previous sections to analyze the effect of flowing near the 
intersection of multiple surfaces of discontinuity generated by a very simple but illustrative family of step functions. 
As noted in \sct{salt}, to describe this effect in general one must solve a collection of variational equations that grows factorially with the number of surfaces of discontinuity.
Thus for clarity in \sct{sync1} and \sct{sync2} we focus on 
a simple family of examples arising from the presence of a generalized signum function. 
We demonstrate that populations of phase oscillators in both first- and second-order versions of this setting can be synchronized via piecewise-constant feedback.

\subsect{Synchronization of First--Order Phase Oscillators}\label{sec:sync1}
In this section we study synchronization in a system consisting of $d$ 
\emph{first--order phase oscillators}, i.e. a control system of the form
\eqnn{\label{eqn:sync1}
\dot{q} = \nu\ones + u(q),
}
where $q\in Q = (S^1)^d$, $\nu\in\R$ is a constant, and $u:Q\into TQ$ is a state--dependent feedback law.
%
The state space is the $d$--dimensional torus $Q = (S^1)^d = \R^d / \Z^d$; we let $\pi:\R^d\into Q$ denote the canonical quotient projection, considered as a covering map~\cite[Appendix~A]{Lee2012}.
In this section, we propose a piecewise--constant form for $u$ and prove that it renders the synchronized orbit 
\eqnn{\label{eqn:sync1:Gamma}
\Gamma = \set{q\in Q \st \forall i,j\in\set{1,\dots,d} : q_i = q_j}
}
locally exponentially stable for~\eqref{eqn:sync1}.

\subsubsect{B--derivative of flow in Euclidean covering space via saltation matrices}
First, we work in the Euclidean covering space, considering the vector field $F:\R^d\into T\R^d$ defined by
\eqnn{\label{eqn:sync1:F}
\forall x\in \R^d : F(x) = \nu\ones - \delta \sgn(x)
}
where $0 < \delta < \nu$ is a given constant
and $\sgn:\R^d\into\cube{d}$ is the vectorized signum function defined as in~\eqref{eqn:sgn}.%
\footnote{We note that there are three common definitions for the scalar signum function, depending on what value one chooses to assign to $0\in\R$, and hence $3^d$ candidate definitions for a vectorized version.  Since integral curves for $F$ spend zero time at the signum's zero crossing, there is no loss of generality in our choice.}
Clearly $F$ is {\EC{\infty}} on $\R^d$ since the event surfaces coincide with the $d$ standard coordinate planes;
for clarity we let $\zerod \in\R^d$ denote the intersection point (i.e. the origin).
Let $\phi:\e{F}\into\R^d$ be the global flow for $F$ yielded by Corollary~\ref{cor:flow}.

We aim to compute the B--derivative of the flow with respect to state along the trajectory passing through $\zerod$.
For clarity we outline the computation here and relegate a detailed derivation to \app{sync1}.
For any word $\word\in\Words$ we can obtain the derivative of the selection function $\phi_\word$ with respect to state from~\eqref{eqn:Dxphiword},
\eqnn{
D_x \phi_\word(0,\zerod) = F(\phi(0,x)) \Lambda^\T_\word(0) + X_\word(0) = \Xi_\word \mat{c}{\zerod^\T \\ I_d},
}
since $\Lambda^\T_\word(0) = \zerod^\T$ and $X_\word(0) = I_d$.
The saltation matrix $\Xi_\omega$, given in general by~\eqref{eqn:salt}, simplifies in this example to~\eqref{eqn:sync1:salt1}, whence we conclude as in~\eqref{eqn:Dxphiword1} that
\eqnn{
\forall\word\in\Words : D_x\phi_\word(0,\zerod) = \frac{\nu-\delta}{\nu+\delta} I_d.
}
This shows that $\phi$ is in fact $C^1$ with respect to state at $(0,\zerod)\in\e{F}$, and hence
\eqnn{\label{eqn:sync1:salt}
\forall w\in T_{\zerod}\R^d : D\phi(0,\zerod; 0,w) = \frac{\nu-\delta}{\nu+\delta} w,
}
i.e. the first--order effect of the nonsmooth flow associated with this piecewise--constant vector field is linear contraction 
at rate $\frac{\nu-\delta}{\nu+\delta}$ independent of the direction $w\in T_{\zerod}\R^d$.

\subsubsect{B--derivative of flow in Euclidean covering space via flowbox}
Before continuing with the task at hand---namely, applying feedback of the form in~\eqref{eqn:sync1:F} to demonstrate synchronization of the first--order phase oscillators in~\eqref{eqn:sync1}---we digress momentarily to provide an alternate derivation of the result in~\eqref{eqn:sync1:salt} that yields additional intuition.
Let $\flbx_0,\flbx_0^{-1}:\R\into\R$ be the piecewise--linear homeomorphisms defined by
\eqnn{\label{eqn:sync1:psi0}
\forall s\in\R : \flbx_0(s) = \pw{s,& s < 0; \\ \frac{\nu+\delta}{\nu-\delta}s,& s \ge 0;}\
\forall \td{s}\in\R :  \flbx_0^{-1}(\td{s}) = \pw{\td{s},& \td{s} < 0; \\ \frac{\nu-\delta}{\nu+\delta}\td{s},& \td{s} \ge 0;}
}
and let $\flbx,\flbx^{-1}:\R^d\into\R^d$ be the piecewise--linear homeomorphisms defined by
\eqnn{\label{eqn:sync1:psi}
\forall x\in\R^d : \flbx(x) = \paren{\flbx_0(x_1), \dots, \flbx_0(x_d)},\
\forall \td{x}\in\R^d : \flbx^{-1}(\td{x}) = \paren{\flbx_0^{-1}(\td{x}_1), \dots, \flbx_0^{-1}(\td{x}_d)}.
}
Note that $\flbx_0\circ\flbx_0^{-1} = \id_\R$ and hence $\flbx\circ\flbx^{-1} = \id_{\R^d}$.
Since furthermore $\flbx\in PC^r(\R^d,\R^d)$, there is no ambiguity in the definition of the ``pushforward'' $\td{F} := D\flbx\circ F\circ\flbx^{-1}:\R^d\into T\R^d$.
In fact, the vector field $\td{F}$ is constant,
\eqnn{
\forall \td{x}\in\R^d : \td{F}(\td{x}) = (\nu+\delta)\ones,
}
and hence its flow $\td{\phi}:\R\times\R^d\into\R^d$ has the simple form
\eqnn{
\forall (t,\td{x})\in\R\times\R^d : \td{\phi}(t,\td{x}) = \td{x} + t (\nu+\delta)\ones.
}
Since the homeomorphism $\flbx$ provides conjugacy between the flows, we have
\eqnn{\label{eqn:sync1:conj}
\forall (t,x)\in\e{F} : \flbx\circ \phi(t,x) = \td{\phi}(t,\flbx(x)) = \flbx(x) +  t (\nu+\delta)\ones;
}
this relationship is illustrated in \fig{box}.
If $t\in\R$ and $x,w\in\R^d$ are such that $x,x+w,x-w\in D_{-\ones}$ and $\phi(t,x),\phi(t,x+w),\phi(t,x-w)\in D_{+\ones}$ as in \fig{box},
the conjugacy in~\eqref{eqn:sync1:conj} can be used to evaluate the B--derivative of the flow $D\phi$, since
\eqnn{
\phi(t,x + sw) & = \flbx^{-1}\paren{\flbx(x + sw) + t(\nu+\delta)\ones} \\
& = \frac{\nu-\delta}{\nu+\delta}\paren{(x + sw) + t(\nu +\delta)\ones} \\
& = \frac{\nu-\delta}{\nu+\delta}(x + sw) + t(\nu -\delta)\ones 
}
and hence
\eqnn{
\lim_{s\goesto 0^+} \frac{1}{s}\paren{\phi(t,x+sw) - \phi(t,x)} = \frac{\nu-\delta}{\nu+\delta} w,\
}
whence~\eqref{eqn:sync1:salt} follows directly.
We emphasize that this outcome---the piecewise--differ-entiable flow is $C^1$ with respect to state---will not arise in general,
but note that other examples in this vein can be obtained by applying other piecewise--linear homeomorphisms to a constant vector field (i.e. a flowbox) so long as the constant vector field is transverse to surfaces of non--smoothness for the homeomorphism (needed to ensure the vector field is {\ECr}).

We conclude by noting that this approach to computing $D\phi$ required a closed--form expression for the ``flowbox'' homeomorphism $\flbx$ and its inverse $\flbx^{-1}$, which is equivalent to possessing a closed--form expression for the flow $\phi$.
Since such expressions are rarely available in applications of interest, in general we expect to rely on the technique developed in \sct{salt} to compute the B--derivative of the flow.

\begin{figure*}[t]
\centering
\subfloat[$\dot{x} = F(x)$]{\label{fig:box1}
\iftoggle{siads}{
\includegraphics[width=.6\columnwidth,viewport=0 0 455 375,clip]{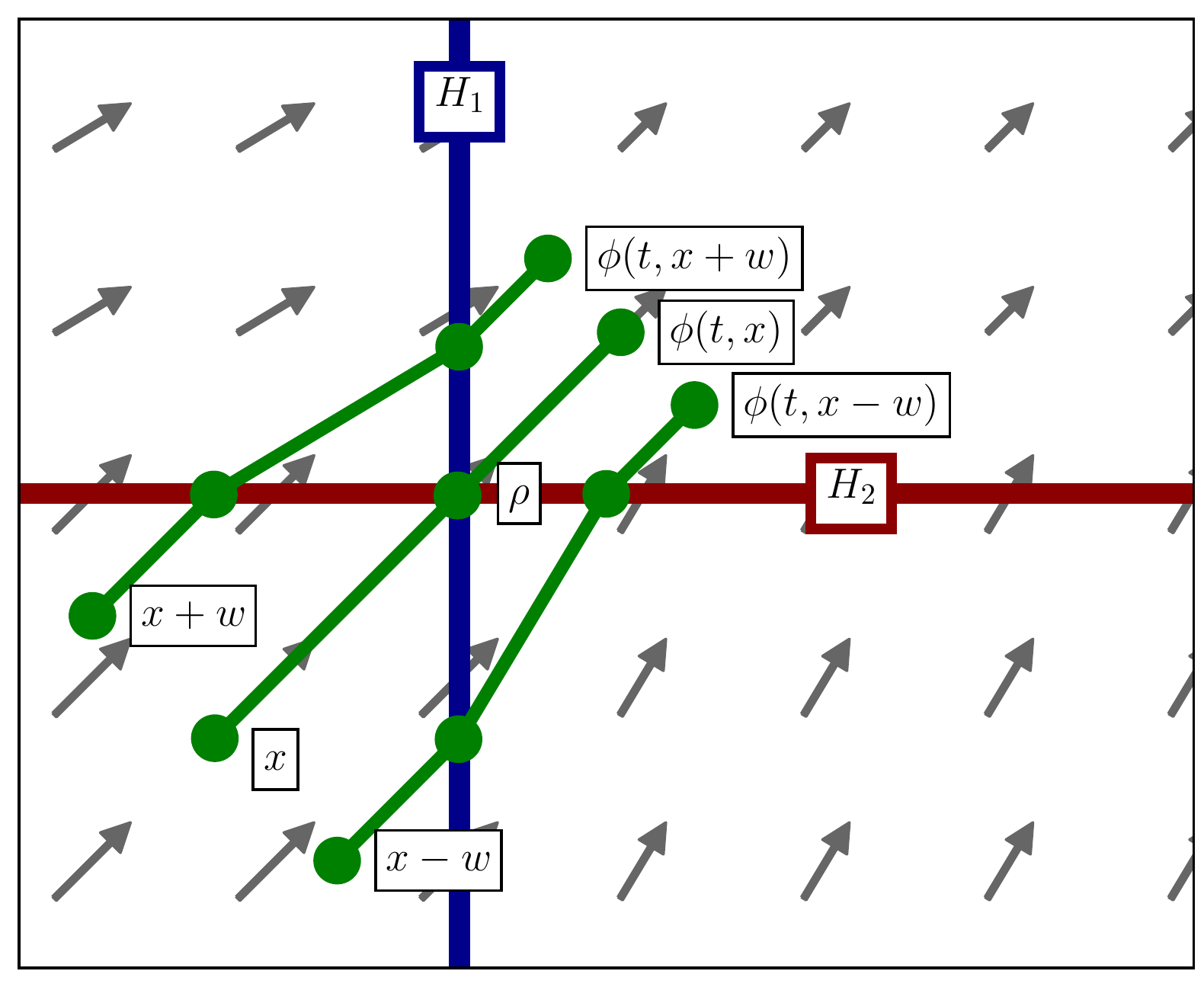}
}{
\includegraphics[width=.47\columnwidth,viewport=0 0 455 375,clip]{fig/box.pdf}
}
}
\iftoggle{siads}{\\}{}
\subfloat[$\dot{\td{x}} = D\flbx\circ F\circ\flbx^{-1}(\td{x})$]{\label{fig:box2}
\iftoggle{siads}{
\includegraphics[width=.6\columnwidth,viewport=535 0 990 375,clip]{fig/box.pdf}
}{
\includegraphics[width=.47\columnwidth,viewport=535 0 990 375,clip]{fig/box.pdf}
}
}
\caption{\label{fig:box}
(a) Illustration of the vector field $F:\R^2\into T\R^2$ from~\eqref{eqn:sync1:F} in the planar case $d = 2$. 
\iftoggle{siads}{\\}{}
(b) Pushforward of $F$ via the piecewise--linear (``flowbox'') homeomorphism $\flbx:D\into D$ from~\eqref{eqn:sync1:psi}.
}
\end{figure*}

\subsubsect{Synchronization via piecewise--constant feedback}
Now, returning to the state space of interest, let $U_\Delta\subset Q$ be the following open set parameterized by $\Delta > 0$:
\eqnn{\label{eqn:sync1:U}
U_\Delta = \set{q\in TQ \mid 
\exists x\in\pi^{-1}(q) : \normo{x} \le \frac{\Delta}{d}
};
}
for $\Delta > 0$ sufficiently small, $U_\Delta$ is ``evenly covered'' in the sense that $\pi|_{\pi^{-1}}(U_\Delta)$ is a homeomorphism~\cite[Appendix~A]{Lee2012}.
Consider the effect of applying feedback of the form 
\eqnn{\label{eqn:sync1:u}
\forall q\in Q : u(q) = \pw{ -\delta\sgn\circ\,\pi^{-1}(q),& q\in U; \\ 0,& q\in Q\sm U;} 
}
to~\eqref{eqn:sync1}. 
It is straightforward to show (as we do in \app{sync1})
that the synchronized orbit $\Gamma$ defined in~\eqref{eqn:sync1:Gamma} is a periodic orbit for~\eqref{eqn:sync1} under this feedback;
we note that the closed--loop dynamics determine an {\EC{\infty}} vector field on a neighborhood of $\Gamma$.

Now we choose a local section $\Sigma\subset Q\sm U_\Delta$ for the closed--loop dynamics that is perpendicular to $\Gamma$ and let $P\in PC^\infty(S,\Sigma)$ denote a Poincar\'{e} map for $\Gamma$ over a neighborhood $S\subset\Sigma$ containing $\set{\fp} = \Gamma\cap\Sigma$. 
To compute $DP(\fp)$ we employ~\eqref{eqn:DpsiDphi}, which involves solving the jump--linear time--varying differential equation~\eqref{eqn:salt:mat} with the saltation matrix update given by~\eqref{eqn:sync1:salt}.
Note that away from discontinuities introduced by the feedback~\eqref{eqn:sync1:u} the vector field in~\eqref{eqn:sync1} does not depend on the state. 
This implies that $D_x F \equiv 0$, hence the continuous--time portion of the variational dynamics~\eqref{eqn:salt:mat} does not alter the derivative computation.

Focusing our attention to the discrete--time (saltation matrix) portion of the variational dynamics~\eqref{eqn:salt:mat},
the closed--loop dynamics are discontinuous
at three points in $\Gamma$:
$\set{-\Delta\ones,\zerod,+\Delta\ones}$.
At $\zerod$, the saltation matrix is given by~\eqref{eqn:sync1:salt}.
At $\pm\Delta\ones$, the update is determined by a single event surface that we chose to be perpendicular to $\Gamma$; although these updates affect $D\phi$, they have no effect on $DP$ since they lie in the kernel of $D\psi$ in~\eqref{eqn:DpsiDphi}.
We conclude that $P$ is $C^1$ and
\eqnn{\label{eqn:sync1:DP}
DP(\fp) = \frac{\nu - \delta}{\nu + \delta} I_{d-1}.
}
Therefore the induced norm contraction hypothesis of~{\propinorm} is satisfied with the standard Euclidean norm and $c = \frac{\nu - \delta}{\nu + \delta}$.
We conclude that $\Gamma$ is exponentially stable, whence the state feedback in~\eqref{eqn:sync1:u} synchronizes the first--order phase oscillators in~\eqref{eqn:sync1} at an exponential rate.

\subsect{Synchronization of Second--Order Phase Oscillators}\label{sec:sync2}
In this section we study synchronization in a system consisting of $d$ \emph{second--order phase oscillators}, i.e. a control system of the form
\eqnn{\label{eqn:sync2}
\ddot{q} = \alpha\ones - \beta\dot{q} + u(q,\dot{q}),
}
where $q\in Q = \R^d/\Z^d$, $\alpha,\beta\in\R$ are constants, and $u:TQ\into T^*Q$ is a state--dependent feedback law.
The state space is the tangent bundle $TQ$ of the $d$--dimension-al torus $Q = \R^d / \Z^d$; we let $\pi:\R^{2d}\into TQ$ denote the canonical quotient projection.

If $u\equiv \mu\ones$ where $\mu\in\R$ is a constant then~\eqref{eqn:sync2} reduces to $d$ decoupled cascades of a pair of scalar affine time--invariant systems, 
thus it is clear that $\ddot{q}\goesto 0$ and hence $\dot{q}\goesto \frac{\alpha + \mu}{\beta}\ones$ as $t\goesto\infty$; this convergence is exponential with rate $\beta$.
In this section, we propose a piecewise--constant form for the feedback $u$ and prove that for all $\beta$ sufficiently large there exists an exponentially stable periodic orbit that passes near $(0,\frac{\alpha}{\beta}\ones)\in TQ$.

\subsubsect{B--derivative of flow in Euclidean covering space}
First, consider the vector field $F:\R^{2d}\into T\R^{2d}$ defined by
\eqnn{\label{eqn:sync2:F}
\forall (x,\dot{x})\in \R^{2d} : F(x,\dot{x}) = \mat{c}{\dot{x} \\ \alpha\ones - \beta \dot{x} - \delta \sgn(x)}
}
where $0 < \delta < \alpha$ is a given constant.
Clearly $F$ is {\EC{\infty}} on the open set
\eqnn{
D = \set{(x,\dot{x})\in\R^{2d} \st \forall j\in\set{1,\dots,d} : \dot{x}_j \ne 0} \subset \R^{2d}
}
since the event surfaces coincide with the first $d$ standard coordinate planes in $\R^{2d}$;
since $F$ fails to be {\ECr} at points with zero velocity, we exclude them from our analysis.
Let $\phi:\e{F}\into D$ denote the global flow for $F$ yielded by Corollary~\ref{cor:flow}.

We begin by computing the B--derivative of the flow with respect to state along the trajectory passing through a point $(0,\nu\ones)\in D$ where $\nu > 0$.
For clarity we outline the computation here and relegate a detailed derivation to \app{sync2}.
For any word $\word\in\Words$ we can obtain the derivative of the selection function $\phi_\word$ with respect to state from~\eqref{eqn:Dxphiword},
\eqnn{
D_x \phi_\word(0,(0,\nu\ones)) = F(\phi(0,x)) \Lambda^\T_\word(0) + X_\word(0) = \Xi_\word \mat{c}{\zerodd^\T \\ I_{2d}},
}
since $\Lambda^\T_\word(0) = \zerodd^\T$ and $X_\word(0) = I_{2d}$.
The saltation matrix $\Xi_\omega$, given in general by~\eqref{eqn:salt}, simplifies in this example to~\eqref{eqn:salt2}, whence we conclude as in~\eqref{eqn:Dxphiword2} that
\eqnn{
D_x\phi_\word(0,(0,\nu\ones)) = \mat{cc}{I_d & 0 \\ -\frac{2\delta}{\nu} I_d & I_d}.
}
This shows that $\phi$ is in fact $C^1$ with respect to state at $(0,(0,\nu\ones))\in\e{F}$, and hence
\eqnn{\label{eqn:sync2:salt}
\forall (p,\dot{p})\in T_{(0,\nu\ones)} D : D\phi(0,(0,\nu\ones); 0,(p,\dot{p})) = \mat{cc}{I_d & 0 \\ -\frac{2\delta}{\nu} I_d & I_d} \mat{c}{p \\ \dot{p}} =: \Xi \mat{c}{p \\ \dot{p}},
}
i.e. the first--order effect of the nonsmooth flow associated with this piecewise--constant vector field is a change in velocity $\dot{p}\mapsto \dot{p} - \frac{2\delta}{\nu} p$ that is proportional to the error in position $p$.
Solving the variational equation as in \sct{salt}, a straightforward calculation (given for completeness in \app{sync2}) yields
\eqnn{\label{eqn:sync2:var}
\mat{c}{p(s) \\ \dot{p}(s)} = \mat{cc}{I_d & \frac{1}{\beta}\paren{1 - e^{-\beta s}} I_d \\ 0 & e^{-\beta s} I_d} \mat{c}{p(0^+) \\ \dot{p}(0^+)} =: X(s) \mat{c}{p(0^+) \\ \dot{p}(0^+)},
}
where $(p(0^+),\dot{p}(0^+))$ is determined from $(p(0),\dot{p}(0))\in T_{(0,\nu\ones)} D$ by applying~\eqref{eqn:sync2:salt}.
Combining~\eqref{eqn:sync2:salt} with~\eqref{eqn:sync2:var} we conclude that the B--derivative with respect to state at time $s$ is given by
\eqnn{
\mat{c}{p(s) \\ \dot{p}(s)} 
& = D\phi\paren{s,\mat{c}{x(0) \\ \dot{x}(0)}; 0, \mat{c}{p(0) \\ \dot{p}(0)}} \\
& = X(s)\ \Xi \mat{c}{p(0) \\ \dot{p}(0)} \\
& = \mat{cc}{I_d & \frac{1}{\beta}\paren{1 - e^{-\beta s}} I_d \\ 0 & e^{-\beta s} I_d} 
\mat{cc}{I_d & 0 \\ -\frac{2\delta}{\nu} I_d & I_d} 
\mat{c}{p(0) \\ \dot{p}(0)} \\
& = \mat{cc}{I_d - \frac{2\delta}{\beta\nu} \paren{1 - e^{-\beta s}} I_d & \frac{1}{\beta}\paren{1 - e^{-\beta s}} \\ -\frac{2\delta}{\nu} e^{-\beta s} I_d & e^{-\beta s} I_d}
\mat{c}{p(0) \\ \dot{p}(0)}.
} 
Taking the limit as $s\goesto\infty$, 
\eqnn{\label{eqn:sync2:lim}
\mat{c}{p(\infty) \\ \dot{p}(\infty)} = \lim_{s\goesto\infty} \mat{c}{p(s) \\ \dot{p}(s)} = \mat{c}{\paren{1 - \frac{2\delta}{\beta\nu}} p(0) \\ 0}.
}
In plain language~\eqref{eqn:sync2:lim} indicates that, to first order, the nonsmooth flow 
associated with the vector field \eqref{eqn:sync2:F}
asymptotically 
(i) drives the initial velocity error $\dot{p}(0)$ to zero 
and
(ii) multiplies the initial position error $p(0)$ by a factor of $c = \paren{1 - \frac{2\delta}{\beta\nu}}$.
If we ensure $\nu \in \paren{\frac{\alpha}{\beta}, \frac{\alpha+\delta}{\beta}}$, then $c = \paren{1 - \frac{2\delta}{\beta\nu}}\in \paren{1 - \frac{2\delta}{\alpha}, 1-\frac{2\delta}{\alpha+\delta}} \subset (-1,+1)$, achieving contraction in positions.
Finally, we note that the convergence in~\eqref{eqn:sync2:lim} is exponential with rate $\beta$.

\subsubsect{Synchronization via piecewise--constant feedback}
We now apply a construction analogous to that of \sct{sync1} to define a piecewise--constant feedback law that results in an exponentially stable periodic orbit that passes near $(0,\frac{\alpha}{\beta}\ones)\in TQ$.
To that end, consider the following form for the control neighborhood $U_\Delta\subset TQ$ parameterized by $\Delta > 0$:
\eqnn{\label{eqn:sync2:U}
U_\Delta = \left\{\vphantom{\int}(q,\dot{q})\in TQ \mid \right.
& \paren{\exists (x,\dot{x})\in\pi^{-1}(q,\dot{q}) : \normo{x} \le \frac{\Delta}{d}}  \\
& \left. \vphantom{\int}
\wedge
\paren{\forall j\in\set{1,\dots,d}: \dot{q}_j > 0}
\right\};
}
for $\Delta > 0$ sufficiently small, $U_\Delta$ is ``evenly covered'' in the sense that $\pi|_{\pi^{-1}}(U_\Delta)$ is a homeomorphism~\cite[Appendix~A]{Lee2012}.
Furthermore, ``synchronized'' points of the form $(\pm\Delta\ones,\nu\ones)$ where $\nu > 0$ are in the boundary $\bd U_\Delta$.
We study the effect of applying feedback of the form 
\eqnn{\label{eqn:sync2:u}
\forall (q,\dot{q})\in TQ : u(q,\dot{q}) = \pw{ -\delta\sgn\circ\,\pi^{-1}(q,\dot{q}),& (q,\dot{q})\in U_\Delta; \\ 0,& (q,\dot{q})\in TQ\sm U_\Delta;} 
}
to~\eqref{eqn:sync2}. 
It is straightforward to show (as we do in \app{sync2}) 
that for all $\beta > 0$ sufficiently large there exists $\nu_\beta \in \paren{\frac{\alpha}{\beta}, \frac{\alpha+\delta}{\beta}}$ such that the trajectory initialized at $(0,\nu_\beta\ones)$ is periodic for the dynamics in~\eqref{eqn:sync2} subject to the piecewise--constant forcing~\eqref{eqn:sync2:u}.
We let
$\Gamma_\beta \subset TQ$
denote the image of the periodic orbit,
and let
$\nu_\beta^-$ 
(resp. $\nu_\beta^+ > 0$) 
denote the speed of the orbit when the position is equal to $-\Delta\ones$ 
(resp. $+\Delta\ones$)
so that
$(-\Delta\ones,\nu_\beta^-\ones)\in\Gamma_\beta$
(resp. $(+\Delta\ones,\nu_\beta^+\ones)\in\Gamma_\beta$)%
.
Note that,
by increasing $\beta$,
$\nu_\beta^-$ can be made arbitrarily close to $\frac{\alpha}{\beta}$ 
and 
$\nu_\beta$ can be made arbitrarily close to $\frac{\alpha+\delta}{\beta}$, 
whence $\nu_\beta\in\paren{\frac{\alpha}{\beta},\frac{\alpha+\delta}{\beta}}$.
Further, note that the closed--loop dynamics determine an {\EC{\infty}} vector field on a neighborhood of $\Gamma_\beta$.

Now we choose a local section $\Sigma_\beta\subset TQ\sm U_\Delta$ for the closed--loop dynamics whose normal vector is parallel to $(\ones,0)$ 
at the point $\fp_\beta = (-\Delta\ones,\nu_\beta^-\ones)\in\Gamma_\beta\cap\Sigma_\beta$.
Note that by construction $\Sigma_\beta\cap\bd U_\Delta$ is an open set containing $\fp_\beta$.
Let $P_\beta\in PC^\infty(S_\beta,\Sigma_\beta)$ denote a Poincar\'{e} map for $\Gamma_\beta$ over a neighborhood $S_\beta\subset\Sigma_\beta$ containing $\set{\fp_\beta}$.
To compute $DP_\beta(\fp_\beta)$ we employ~\eqref{eqn:DpsiDphi}, which involves solving the jump--linear time--varying differential equation~\eqref{eqn:salt:mat} with the saltation matrix update given by~\eqref{eqn:sync2:salt}.
Away from discontinuities introduced by the feedback~\eqref{eqn:sync2:u} the state dependence of the vector field in~\eqref{eqn:sync2} is confined to viscous drag on velocities.
This implies that the continuous--time portion of the variational dynamics~\eqref{eqn:salt:mat} is given by~\eqref{eqn:sync2:var}, i.e. the first--order effect of the flow contracts velocity error at an exponential rate and amplfies position error by an amount proportional to $1/\beta$. 

Focusing our attention now on the discrete--time (saltation matrix) portion of the variational dynamics~\eqref{eqn:salt:mat},
the closed--loop dynamics are discontinuous
at three points in $\Gamma$:
$(-\Delta\ones,\nu_\beta^-\ones)$,
$(0,\nu\ones)$,
and
$(+\Delta\ones,\nu_\beta^+\ones)$.
At $(0,\nu\ones)$, the saltation matrix is given by~\eqref{eqn:sync2:salt}.
At $(\pm\Delta\ones,\nu_\beta^\pm)$, the saltation matrix is determined by a single event surface whose normal vector is parallel to $(\ones,0)$. 
Although these updates affect $D\phi$, they have no effect on $DP_\beta$ since they lie in the kernel of $D\psi$ in~\eqref{eqn:DpsiDphi}.
We conclude that $P_\beta$ is $C^1$ and
\eqnn{\label{eqn:sync2:DP}
DP_\beta(\fp_\beta) = \mat{cc}{\paren{1 - \frac{2\delta}{\beta\nu_\beta}} I_{d-1} & 0 \\ 0 & 0} + E_\beta 
}
where the induced norm of the error term $\inorm{E_\beta}$ decreases exponentially with increasing $\beta$.
Therefore for all $\beta > 0$ sufficiently large the induced norm contraction hypothesis of~{\propinorm} is satisfied with the standard Euclidean norm and 
$c \approx \paren{1 - \frac{2\delta}{\beta\nu_\beta}}\in \paren{1 - \frac{2\delta}{\alpha}, 1-\frac{2\delta}{\alpha+\delta}} \subset (-1,+1)$.
We conclude that $\Gamma_\beta$ is exponentially stable for all $\beta > 0$ sufficiently large, whence the state feedback in~\eqref{eqn:sync2:u} synchronizes the second--order phase oscillators in~\eqref{eqn:sync2} at an exponential rate.

\sect{Discussion}\label{sec:disc}

In this paper, we studied local properties of the flow generated by vector fields with ``event--selected'' discontinuities, that is, vector fields that are 
(i) smooth except along a finite collection of smooth submanifolds and 
(ii) ``transverse'' to these submanifolds in the sense that integral curves intersect them at isolated points in time.
We emphasize that the vector field transversality condition (ii) excludes \emph{sliding modes}~\cite{Utkin1977, Jeffrey2014siads} from our analysis.
Basic properties of discontinuous vector fields have been studied in a more general setting, for instance yielding sufficient conditions ensuring existence of a continuous flow (see~\cite[Chapter~2]{Filippov1988} generally and~\cite[Theorem~3 in~\S8]{Filippov1988} specifically).
Our chief contribution is the introduction of techniques from non--smooth analysis~\cite{Scholtes2012} to show that a vector field with event--selected discontinuities yields a continuous flow that admits a strong first--order approximation, the (so--called~\cite{Robinson1987}) \emph{Bouligand} derivative. 
We employed this B--derivative to obtain fundamental constructions familiar from classical (smooth) dynamical systems theory, including
impact maps, flowboxes, and variational equations,
and to study the effect of perturbations, both infinitesimal and non--infinitesimal.
In the classical setting, these constructions are obtained using the classical (alternately called \Frechet~\cite[Section~3.1]{Scholtes2012} or \Jacobian~\cite[Section~1.3]{GuckenheimerHolmes1983}) derivative of the smooth flow;
our construction of the non--smooth object proceeded analogously to that of its smooth counterpart after replacing the classical derivative of the flow with our B--derivative.
Thus the piecewise--differentiable dynamical systems we study bear a closer resemblance to classically differentiable dynamical systems than to discontinuous dynamical systems considered, for instance, in~\cite{PringBudd2010, JimenezMihalas2013}.

In future work, we expect to obtain generalizations of other techniques from the classical theories of dynamical and control systems 
that depend primarily on the existence of first-- or higher--order approximations of the flow, for instance:
stability analysis via 
control Lyapunov functions~\cite{Artstein1983, Sontag1989artstein}
or infinitesimal contractivity~\cite{LohmillerSlotine1998, Sontag2010};
conditions for controllability based on the inverse function theorem~\cite[Theorem~8]{LevinNarendra1993}~\cite[Section~II.C]{CarverCowan2009};
or
necessary and sufficient conditions for optimality in nonlinear programs involving dynamical systems~\cite[Chapter~4]{Polak1997}.
More broadly, we believe our results support the study of a class of discontinuous vector fields that arise in 
neuroscience~\cite{KeenerHoppensteadt1981},
biological and robotic locomotion~\cite{HolmesFull2006}, 
and 
electrical engineering~\cite{Hiskens1995}.
In each of these disparate domains, behaviors of interest occur near the intersection of surfaces of discontinuity, hence the techniques we developed in this paper may be brought to bear.
Thus we conclude with brief remarks about the formal applicability and practical relevance of our results in these applications.

Integrate--and--fire neuron models consist of a population of $n$ subsystems that undergo a discontinuous change in membrane voltage triggered by crossing a voltage threshold~\cite{KeenerHoppensteadt1981},  
\eqnn{\label{eqn:inf}
\dot{v} & = -\gamma v + u,\\
v(t^+) & = 0\ \text{if}\ v(t^-) = \obar{v},
}
where: 
$v\in\R$ is the membrane voltage; 
$\gamma\in\R$ is a dissipation constant; 
$u\in\R$ is an exogenous input; 
and 
$\obar{v}\in\R$ is the firing threshold.
When driven by a periodic exogenous input, integrate--and--fire neuron populations can exhibit \emph{phase locking}~\cite{KeenerHoppensteadt1981} or \emph{local synchronization}~\cite{HopfieldHerz1995} behavior, resulting in simultaneous or near--simultaneous firing.
Of interest in applications is the computation of so--called%
\footnote{In practice, one computes singular values of finite--time sensitivity matrices, rather than the formal asymptotic Lyapunov exponent as defined, for instance, in~\cite[Section~3.4.1]{Sastry1999}.}
``Lyapunov exponents'' using variational equations.
We showed in~\sct{salt} that the variational equation must be supplemented by discontinuous updates via saltation matrices near such simultaneous--firing events.
As noted in~\cite[Section~4.1]{BizzarriBrambilla2013}, neglecting this non--smooth effect can result in erroneous conclusions.

Legged locomotion of animals and robots involves intermittent interaction of limbs with terrain;
their dynamics are given by~\cite[Section~II]{JohnsonBurden2015},
\eqnn{\label{eqn:sm}
M(q)\ddot{q} & = f(q,\dot{q}) + \lambda(q,\dot{q}) Da(q),\\
\dot{q}(t^+) & = R(q(t)) \dot{q}(t^-)\ \text{if}\ a_j(t^-) = 0,
}
where:
$q\in Q$ is a vector of generalized coordinates for the body and limbs; 
$M$ is the inertia tensor;
$f:TQ\into T^*Q$ contains internal, applied, and Coriolis forces;
$a:Q\into\R^{n}$ specifies $n$ unilateral constraints of the form 
\eqnn{\label{eqn:sm:uni}
\forall j\in\set{1,\dots,n} : a_j(q)\ge 0;
}
$\lambda:TQ\into T^*Q$ denotes the reaction forces that ensure~\eqref{eqn:sm:uni} are satisfied by~\eqref{eqn:sm} for all time.
The update $\dot{q}(t^+) = R(q(t)) \dot{q}(t^-)$ is triggered when one of the unilateral constraints $a_j$ would be violated by a penetrating velocity; 
it causes a discontinuity in both the velocity and the forces acting on the system.
Legged animals and robots with four, six, and more limbs exhibit gaits with near--simultaneous touchdown of two or more legs~\cite{Alexander1984, GolubitskyStewart1999, HolmesFull2006}. 
Steady--state gaits are commonly modeled as periodic orbits in the body reference frame~\cite{SchonerKelso1988, KoditschekBuhler1991, KubowFull1999}.
In practice, gait stability is assessed using the linearization of the first--return or {\Poincare} map, since if the eigenvalues of the linearization (the so--called \emph{Floquet multipliers}~\cite[Section~1.5]{GuckenheimerHolmes1983}) lie within the unit disk then the gait is exponentially stable~\cite{AizermanGantmacher1958, GrizzleAbba2002}.
We showed in~\sct{Pmap} that the {\Poincare} map associated with a periodic orbit passing through the intersection of multiple surfaces of discontinuity is generally non--smooth.
This implies that
it is not possible to assess stability of such orbits using the F--derivative
since the F--derivative of the map does not exist.
In~\sct{stab}, we showed how the B--derivative of the {\Poincare} map can be employed instead to assess stability of such orbits.

Electrical power systems undergo discontinuous changes in network topology triggered by excessive voltages or currents,
leading to differential--algebraic models of the form~\cite[(2)--(4)]{Hiskens1995}
\eqnn{\label{eqn:pow}
\dot{x} & = f(x,y,z;p),\ 0 = g(x,y,z;p),\\
z(t^+) &= h(x(t),y(t),z(t^-);p)\ \text{if}\ y_j(t^-) = 0,
}
where: 
$x\in\R^d$ contains dynamic states; 
$y\in\R^n$ contains algebraic states; 
$z\in\R^m$ contains discrete states; 
$p\in\R^\ell$ contains parameters;
$f:\R^{d+n+m+\ell}\into\R^d$ is a smooth vector field;
$g:\R^{d+n+m+\ell}\into\R^k$ is a smooth constraint function;
and
$h:\R^{d+n+m+\ell}\into\R^m$ is a smooth reset function. 
The update $z(t^+) = h(x(t),y(t),z(t^-);p)$ is applied when one of the algebraic states $y_j$ crosses a prespecified threshold (e.g. a bus voltage limit), causing a discontinuity in the vector field governing the time evolution of $x$.
In electrical power networks, discrete switches triggered by over--excitation limits can occur at arbitrary times with respect to one another.
When the switches occur at distinct time instants, the trajectory sensitivity matrix (i.e. the F--derivative the flow) computed as in~\cite{HiskensPai2000} can provide quantitative insights for design and control.
However, as noted in~\cite[Section~VIII]{HiskensPai2000}, these calculations lose accuracy when event times become coincident; this is due to the fact that the flow is not classically differentiable along trajectories that undergo simultaneous discrete transitions.
The procedure we developed in~\sct{salt} can be employed to compute a collection of trajectory sensitivity matrices (i.e. the B--derivative of the flow) that generalize the approach advocated in~\cite{HiskensPai2000} to be applicable in power networks that undergo an arbitrary (but finite) number of simultaneous discrete transitions.

\subsection*{Support}
This research was supported in part by: an NSF Graduate Research Fellowship to S. A. Burden;
ARO Young Investigator Award \#61770 to S. Revzen;
ARL Cooperative Agreements W911NF--08--2--0004 and W911NF--10--2--0016;
and NSF Award \#1028237.
The views and conclusions contained in this document are those of the authors and should not be interpreted as representing the official policies, either expressed or implied, of the Army Research Laboratory or the U.S. Government.  
The U.S. Government is authorized to reproduce and distribute for Government purposes notwithstanding any copyright notation herein.

\appendix

\sect{F--Derivative Formulae for Piecewise--Differentiable Flow}\label{app:D}
A piecewise--differentiable function is differentiable almost everywhere~\cite[Theorem~2]{Rockafellar2003}, and hence its B--derivative at any point is contained in the convex hull of the limit of F--derivatives of its selection functions~\cite[\S~4.3]{Scholtes2012}.
For completeness and to aid the reader's comprehension, we now derive explicit formulae for F--derivatives of the piecewise--differentiable objects used in \sct{flow} to construct the piecewise--differentiable flow.
In general the B--derivative can be obtained via the chain rule~\cite[Theorem~3.1.1]{Scholtes2012}.

\subsect{Budgeted time--to--boundary}\label{app:bttb}
We adopt the notational conventions from \sct{bttb}.

Define $\nu_b^+:U_b\into\R\cup\set{+\infty}$ using the convention $\min\emptyset = +\infty$ by 
\eqnn{\label{eqn:nubp}
\forall x\in U_b : \nu_b^+(x) = \min\set{\tau_b^{H_j}(x) : b_j < 0}_{j=1}^n,
}
then for all $(t,x)\in\R\times U_b$ such that $\nu_b^+(x) \ne t \ne 0$, 
the forward--time budgeted time--to--boundary 
$\tau_b^+$ is classically differentiable and
\eqnn{
D\tau_b^+(t,x) = \pw{(0,\ \zerod^\T),& t < 0; \\ (1,\ \zerod^\T),& 0 < t < \nu_b^+(x); \\ \paren{0, D\tau_b^H(x)},& \nu_b^+(x) < t;}
}
where in the third case $H\in\set{H_j}_{j=1}^n$ is such that $\tau_b^H(x) = \nu_b^+(x) > 0$.

Define $\nu_b^-:U_b\into\R\cup\set{-\infty}$ using the convention $\max\emptyset = -\infty$ by 
\eqnn{\label{eqn:nubm}
\forall x\in U_b : \nu_b^-(x) = \min\set{\tau_b^{H_j}(x) : b_j > 0}_{j=1}^n,
}
then for all $(t,x)\in\R\times U_b$ such that $\nu_b^-(x) \ne t \ne 0$, 
the backward--time budgeted time--to--boundary 
$\tau_b^-$ is classically differentiable and
\eqnn{\label{eqn:Dtaubm}
D\tau_b^-(t,x) = \pw{(0,\ \zerod^\T),& t > 0; \\ (1,\ \zerod^\T),& \nu_b^-(x) < t < 0; \\ \paren{0, D\tau_b^H(x)},& t < \nu_b^-(x);}
}
where in the third case $H\in\set{H_j}_{j=1}^n$ is such that $\tau_b^H(x) = \nu_b^-(x) < 0$.

\subsect{Budgeted flow--to--boundary}\label{app:ftb}
We adopt the notational conventions from \sct{ftb}.

Define $\nu_b^+:U_b\into\R$ as in~\eqref{eqn:nubp} then for all $(t,x)\in\R\times U_b$ such that $\nu_b^+(x) \ne t \ne 0$, the forward--time flow--to--boundary $\zeta_b^+$ is classically differentiable and
\eqnn{
D\zeta_b^+(t,x) = \pw{(\zerod,\ \zerodxd),& t < 0; \\ (F_b(\phi_b(t,x)),\ D_x\phi_b(t,x)),& 0 < t < \nu_b^+(x); \\ (\zerod,\Upsilon(t,x)),& \nu_b^+(x) < t;}
}
where in the third case $\Upsilon(t,x) = F_b(\phi_b(\tau_b^+(t,x),x)) D\tau_b^H(x) + D_x\phi_b(\tau_b^+(t,x),x)$ and $H\in\set{H_j}_{j=1}^n$ is such that $\tau_b^H(x) = \nu_b^+(x)$.

Define $\nu_b^-:U_b\into\R$ as in~\eqref{eqn:nubm} then for all $(t,x)\in\R\times U_b$ such that $\nu_b^-(x) \ne t \ne 0$, the backward--time flow--to--boundary $\zeta_b^-$ is classically differentiable and
\eqnn{\label{eqn:Dzetabm}
D\zeta_b^-(t,x) = \pw{(\zerod,\ \zerodxd),& t > 0; \\ (F_b(\phi_b(t,x)),\ D_x\phi_b(t,x)),& \nu_b^-(x) < t < 0; \\ \paren{\zerod,\Upsilon(t,x) },& t < \nu_b^-(x);}
}
where in the third case $\Upsilon(t,x) = F_b(\phi_b(\tau_b^-(t,x),x)) D\tau_b^H(x) + D_x\phi_b(\tau_b^-(t,x),x)$ and $H\in\set{H_j}_{j=1}^n$ is such that $\tau_b^H(x) = \nu_b^+(x)$.

\subsect{Composite of budgeted time--to-- and flow--to--boundary}\label{app:bttbftb}
We adopt the notational conventions from \sct{bttbftb}.

Combine~\eqref{eqn:Dtaubp} and~\eqref{eqn:Dzetabp} to obtain the derivative of $\vphi_b^+$ for all $(t,x)\in\R\times U_b$ such that $\nu_b^+(x) \ne t \ne 0$:
\eqnn{
D\vphi_b^+(t,x) = \pw{\mat{cc}{1 & 0 \\ 0 & I},& t < 0; \\ \mat{cc}{0 & 0 \\ F_b(\phi_b(t,x)) & D_x\phi_b(t,x)},& 0 < t < \nu_b^+(x); \\ \mat{cc}{1 & -D\tau_b^H(x) \\ 0 & \Upsilon(t,x)},& \nu_b^+(x) < t;}
}
where in the third case $\Upsilon(t,x) = F_b(\phi_b(\tau_b^+(t,x),x)) D\tau_b^H(x) + D_x\phi_b(\tau_b^+(t,x),x)$ and $H\in\set{H_j}_{j=1}^n$ is such that $\tau_b^H(x) = \nu_b^+(x)$.

Combine~\eqref{eqn:Dtaubm} and~\eqref{eqn:Dzetabm} to obtain the derivative of $\vphi_b^-$ for all $(t,x)\in\R\times U_b$ such that $\nu_b^-(x) \ne t \ne 0$:
\eqnn{\label{eqn:Dvphibm}
D\vphi_b^-(t,x) = \pw{\mat{cc}{1 & 0 \\ 0 & I},& t > 0; \\ \mat{cc}{0 & 0 \\ F_b(\phi_b(t,x)) & D_x\phi_b(t,x)},& \nu_b^-(x) < t < 0; \\ \mat{cc}{1 & -D\tau_b^H(x) \\ 0 & \Upsilon(t,x)},& t < \nu_b^-(x);}
}
where in the third case $\Upsilon(t,x) = F_b(\phi_b(\tau_b^-(t,x),x)) D\tau_b^H(x) + D_x\phi_b(\tau_b^-(t,x),x)$ and $H\in\set{H_j}_{j=1}^n$ is such that $\tau_b^H(x) = \nu_b^+(x)$.

\sect{Periodic Orbits and B--derivative Formulae for Phase Oscillators}
In the following subsections we provide detailed derivations that were deemed too laborious to include in the main text.

\subsect{Synchronization of First--Order Phase Oscillators}\label{app:sync1}
We adopt the notational conventions of \sct{sync1}, and 
derive several useful properties of closed--loop dynamics obtained by applying the piecewise--constant feedback in~\eqref{eqn:sync1:u} to the system in~\eqref{eqn:sync1}.

First we argue that the synchronized set~\eqref{eqn:sync1:Gamma} is a periodic orbit for the closed--loop dynamics.
Since~\eqref{eqn:sync1} consists of $d$ identical subsystems and the feedback in~\eqref{eqn:sync1:u} encounters discontinuities synchronously (i.e. all coordinates of the vector field change discontinuously at the same time) at points $\set{-\Delta\ones,\zerod,+\Delta\ones}$ and at every other point in time the vector field coordinates are identical,
we conclude that the trajectory initialized at $\zerod$ remains synchronized for all time.

We now explicitly derive the B--derivative in \eqref{eqn:sync1:salt}.
Fixing a word $\word:\set{1,\dots,d}\into\cube{d}$ with corresponding sequence of surfaces crossed 
$\eta:\set{1,\dots,d}\into\set{1,\dots,d}$, we know from \sct{salt} that for all $\veps > 0$ the F--derivative of the selection function $\phi_\word$ at $(2\veps,\phi(-\veps,\zerod))$ is given by~\eqref{eqn:Dphi},
\eqnn{\label{eqn:sync1:Dphi}
D\phi_\word(2\veps,\phi(-\veps,\zerod)) = D\phi(\veps, \zerod)\brak{\prod_{j=1}^d D\vphi_{\word(j)}^+(0,\zerod)}\mat{c}{0 \\ D\phi(\veps,\phi(-\veps,\zerod))}.
}
Here, $D\phi(\veps, \zerod)$, $D\phi(\veps,\phi(-\veps,\zerod))$ are obtained as in~\eqref{eqn:DtDxphi} by solving the classical variational equation since $F:\R^d\into T\R^d$ is smoothly extendable to a neighborhood of those segments of the trajectory;
noting that 
for all $s\in[-\veps,0)$ we have $F(\phi(s,\zerod)) = F_{-\ones}(\zerod)$, $D_x F(\phi(s,\zerod)) = 0$ and 
for all $s\in(0,+\veps]$ we have $F(\phi(s,\zerod)) = F_{+\ones}(\zerod)$, $D_x F(\phi(s,\zerod)) = 0$,
we conclude that
\eqnn{
D\phi(s, \zerod) = \mat{cc}{F_{+\ones}(\zerod) & I_d},\ 
D\phi(s, \phi(-s,\zerod)) = \mat{cc}{F_{-\ones}(\zerod) & I_d}.
}
For each $j\in\set{1,\dots,d}$ the derivative $D\vphi_{\word(j)}^+(0,\zerod)$ is given by the matrix in the third case in~\eqref{eqn:Dvphibp} with the simplifications $\tau_{\word(j)}^+(0,\zerod) = 0$, $\phi_{\word(j)}(0,\zerod) = \zerod$;
setting $f_j = F_{\word(j)}(\zerod)$, $g_j^\T = Dh_{\eta(j)(\zerod)}$ for each $j\in\set{1,\dots,d}$, by~\eqref{eqn:Dvphibps} we have
\eqnn{
D\vphi_{\word(j)}^+(0,\zerod) = I_{d+1} + \frac{1}{g_j^\T f_j} \mat{c}{1 \\ -f_j} \mat{cc}{0 & g_j^\T}.
}
Now for all $j\in\set{1,\dots,d}$ we have
$f_j = F_{\word(j)}(\zerod) = \nu\ones - \delta \word(j)\in\R^d$ 
and
$g_j^\T = Dh_{\eta(j)}(\zerod) = e_{\eta(j)}^\T\in\R^{1\times d}$,
hence
$g_j^\T f_j = \nu + \delta$.
Since 
for all $i\in\set{1,\dots,d}$ with $i > j$ 
the vector $\word(i)$ is lexicographically greater than $\word(j)$,
we also have 
$g_i^\T f_j = \nu + \delta$.
These simplifications yield for all $j\in\set{1,\dots,d-1}$
\eqnn{
D&\vphi_{\word(j+1)}^+(0,\zerod) 
 D\vphi_{\word(j)}^+(0,\zerod) \\
& = \brak{I_{d+1} + \frac{1}{\nu+\delta}\mat{c}{1 \\ -f_{j+1}} \mat{cc}{0 & e_{\eta(j+1)}^\T}}
\brak{I_{d+1} + \frac{1}{\nu+\delta}\mat{c}{1 \\ -f_{j}} \mat{cc}{0 & e_{\eta(j)}^\T}} \\
& = I_{d+1} + \frac{1}{\nu+\delta}\mat{c}{1 \\ -f_{j+1}} \mat{cc}{0 & e_{\eta(j+1)}^\T}
+ \frac{1}{\nu+\delta}\mat{c}{0 \\ f_{j+1} - f_j} \mat{cc}{0 & e_{\eta(j)}^\T}.
}
Noting that $f_{j+1} - f_j = -2\delta e_{\eta(j)}$,
we conclude 
that $e_{\eta(i)}^\T(f_{j+1}-f_j) = 0$
for all $i\in\set{j+1,\dots,d}$.
This implies that
\eqnn{\label{eqn:sync1:salt1}
\Xi_\word & = \prod_{j=1}^d D\vphi_{\word(j)}^+(0,\zerod) \\
& = I_{d+1} + \frac{1}{\nu+\delta}\mat{c}{1 \\ -f_{d}} \mat{cc}{0 & e_{\eta(d)}^\T}
+ \frac{1}{\nu+\delta}\sum_{j=1}^{d-1} \mat{c}{0 \\ -2\delta e_{\eta(j)}} \mat{cc}{0 & e_{\eta(j)}^\T}.
}
Noting for any $f\in\R^d$ with $0_d = 0\cdot\ones_d\in\R^d$ that
\eqnn{
\mat{cc}{0 & 0_d^\T \\ f & I_d} = I_{d+1} + \mat{c}{-1 \\ f}\mat{cc}{1 & 0_d^\T},
}
defining $f^+ = F_{+\ones}(\zerod)$ and noting that $f^+ - f_d = -2\delta e_{\eta(d)}$ we have
\eqnn{
\mat{cc}{0 & 0_d^\T \\ f^+ & I_d}
\prod_{j=1}^d D\vphi_{\word(j)}^+(0,\zerod) 
= &
\, I_{d+1} 
+ \mat{c}{-1 \\ f^+}\mat{cc}{1 & 0_d^\T} \\
& + \frac{1}{\nu+\delta}\mat{c}{0 \\ f^+ - f_d}\mat{cc}{0 & e_{\eta(d)}^\T} \\ 
& + \frac{1}{\nu+\delta}\sum_{j=1}^{d-1} \mat{c}{0 \\ -2\delta e_{\eta(j)}} \mat{cc}{0 & e_{\eta(j)}^\T} \\
= & \, I_{d+1} - \frac{2\delta}{\nu+\delta}\mat{cc}{0 & 0_d^\T \\ 0_d & I_d} + \mat{c}{-1 \\ f^+}\mat{cc}{1 & 0_d^\T}.
}
Finally, defining $f^- = F_{-\ones}(\zerod)$ we have
\eqnn{
D\phi_\word(0,\zerod) & = \mat{cc}{0 & 0_d^\T \\ f^+ & I_d}
\prod_{j=1}^d D\vphi_{\word(j)}^+(0,\zerod) 
\mat{cc}{0 & 0_d^\T \\ f^- & I_d} \\
& = 
I_{d+1} 
- \frac{2\delta}{\nu+\delta}\mat{cc}{0 & 0_d^\T \\ 0_d & I_d} 
+ \mat{c}{-1 \\ \paren{\frac{\nu-\delta}{\nu+\delta}}f^-}\mat{cc}{1 & 0_d^\T} \\
& = \mat{cc}{\paren{\frac{\nu-\delta}{\nu+\delta}}f^- & \frac{\nu-\delta}{\nu+\delta} I_d}.
}
Restricting to the derivative with respect to state, we find for all words $\word\in\Words$ that
\eqnn{\label{eqn:Dxphiword1}
D_x\phi_\word(0,\zerod) = \frac{\nu-\delta}{\nu+\delta} I_d.
}
Thus the piecewise--differentiable flow $\phi$ is $C^1$ with respect to state at $(0,\zerod)$, and~\eqref{eqn:sync1:salt} follows.
Restricting instead to the derivative with respect to time, we find as expected that
\eqnn{\label{eqn:Dtphiword1}
D_t\phi_\word(0,\zerod) = \frac{\nu-\delta}{\nu+\delta} f^- = f^+ = F_{+\ones}(\zerod).
}

\subsect{Synchronization of Second--Order Phase Oscillators}\label{app:sync2}
We adopt the notational conventions of \sct{sync2}, and 
derive several useful properties of closed--loop dynamics obtained by applying the piecewise--constant feedback in~\eqref{eqn:sync2:u} to the system in~\eqref{eqn:sync2}.

First we argue that, for any $\beta,\Delta > 0$ there exists $\nu_\beta \in \paren{\frac{\alpha}{\beta}, \frac{\alpha+\delta}{\beta}}$ such that the trajectory initialized at $(0,\nu\ones)$ is a periodic orbit for the closed--loop dynamics.
Since~\eqref{eqn:sync2} consists of $d$ identical subsystems and the feedback in~\eqref{eqn:sync2:u} encounters discontinuities synchronously (i.e. all coordinates of the vector field change discontinuously at the same time) at points of the form $(\theta\ones,\nu\ones)$ where $\theta\in\set{-\Delta,0,+\Delta}$ and $\nu > 0$ and at every other point in time the vector field coordinates are identical,
we conclude that a trajectory initialized at $(0,\nu\ones)$ remains synchronized for all time, so the asymptotic behavior of this trajectory can be studied by restricting our attention to the scalar case (i.e. $d = 1$), wherein the dynamics take the simple form
\eqnn{\label{eqn:sync2:d1}
d = 1 \implies \ddot{q} = \pw{\alpha - \delta - \beta \dot{q},\ q\in [-\Delta,0); \\ \alpha + \delta - \beta \dot{q},\ q\in[0,+\Delta]; \\ \alpha - \beta\dot{q},\ \text{else};}
}
here we adopt the abuse of notation that $q\in[-\Delta,0)$ if there exists $x\in[-\Delta,0)\subset\R$ such that $\pi(x) = q$, and similarly for $q\in[0,+\Delta]$.
Clearly if the initial velocity $\dot{q}(0) > 0$ then $\dot{q}(t) > 0$ for all $t > 0$ since $0 < \delta < \alpha$.
This implies that $q(t)$ crosses the thresholds $\theta\in\set{-\Delta,0,+\Delta}$ in sequence.
The impact map $P_{\beta}^{(\theta_1,\theta_2)}:(0,\infty)\into(0,\infty)$
obtained by integrating the flow of~\eqref{eqn:sync2:d1} between any sequential pair of event surfaces $(\theta_1,\theta_2)\in\set{(-\Delta,0),(0,+\Delta),(+\Delta,-\Delta)}$ 
is a contraction over velocities with a Lipschitz constant that decreases exponentially with increasing $\beta$.
Thus the composition $P_\beta = P_\beta^{(0,+\Delta)}\circ P_\beta^{(-\Delta,0)}\circ P_\beta^{(+\Delta,-\Delta)}$ is a contraction for all $\beta$ sufficiently large.
Since furthermore for all $\beta$ sufficiently large the compact set $\brak{\frac{\alpha}{\beta},\frac{\alpha+\delta}{\beta}}$ is mapped to itself under $P_\beta$, the~{\CMT} implies there exists $\nu_\beta\in\paren{\frac{\alpha}{\beta},\frac{\alpha+\delta}{\beta}}$ such that $P_\beta(\nu_\beta) = \nu_\beta$.
In other words, the trajectory initialized at $(0,\nu_\beta)$ lies on a periodic orbit for~\eqref{eqn:sync2:d1}, and hence $(0,\nu_\beta\ones)$ lies on a periodic orbit for the closed--loop dynamics obtained by applying the feedback in~\eqref{eqn:sync2:u} to the system in~\eqref{eqn:sync2}.
It is straightforward to verify in this scalar system that solving the variational equation as in \sct{salt} yields
\eqnn{\label{eqn:sync2:var:app}
\mat{c}{p(s) \\ \dot{p}(s)} = \mat{cc}{1 & \frac{1}{\beta}\paren{1 - e^{-\beta s}} \\ 0 & e^{-\beta s}} \mat{c}{p(0^+) \\ \dot{p}(0^+)} =: X(s) \mat{c}{p(0^+) \\ \dot{p}(0^+)}.
}
Since the saltation updates are synchronous along the periodic orbit,~\eqref{eqn:sync2:var} follows.

We now explicitly derive the B--derivative in \eqref{eqn:sync2:salt}.
Fixing a word $\word:\set{1,\dots,d}\into\cube{d}$ with corresponding sequence of surfaces crossed 
$\eta:\set{1,\dots,d}\into\set{1,\dots,d}$, we know from \sct{salt} that for all $\veps > 0$ the F--derivative of the selection function $\phi_\word$ at $(2\veps,\phi(-\veps,(0,\nu\ones)))$ is given by~\eqref{eqn:Dphi},
\eqn{
D\phi_\word&(2\veps,\phi(-\veps,(0,\nu\ones))) \\
& = D\phi(\veps, (0,\nu\ones))\brak{\prod_{j=1}^d D\vphi_{\word(j)}^+(0,(0,\nu\ones))}\mat{c}{0 \\ D\phi(\veps,\phi(-\veps,(0,\nu\ones)))}.
}
Here, $D\phi(\veps, (0,\nu\ones))$, $D\phi(\veps,\phi(-\veps,(0,\nu\ones)))$ are obtained as in~\eqref{eqn:DtDxphi} by solving the classical variational equation since $F:\R^{2d}\into T\R^{2d}$ is smoothly extendable to a neighborhood of those segments of the trajectory;
we conclude that
\eqnn{
\lim_{s\goesto 0^+}D\phi(s, (0,\nu\ones)) &= \mat{cc}{F_{+\ones}(0,\nu\ones) & I_d},\\ 
\lim_{s\goesto 0^+}D\phi(s, \phi(-s,(0,\nu\ones))) &= \mat{cc}{F_{-\ones}(0,\nu\ones) & I_d}.
}
For each $j\in\set{1,\dots,d}$ the derivative $D\vphi_{\word(j)}^+(0,(0,\nu\ones))$ is given by the matrix in the third case in~\eqref{eqn:Dvphibp} with the simplifications $\tau_{\word(j)}^+(0,(0,\nu\ones)) = 0$, $\phi_{\word(j)}(0,(0,\nu\ones)) = (0,\nu\ones)$;
setting $f_j = F_{\word(j)}(0,\nu\ones)$, $g_j^\T = Dh_{\eta(j)(0,\nu\ones)}$ for each $j\in\set{1,\dots,d}$, by~\eqref{eqn:Dvphibps} we have
\eqnn{
D\vphi_{\word(j)}^+(0,(0,\nu\ones)) = I_{d+1} + \frac{1}{g_j^\T f_j} \mat{c}{1 \\ -f_j} \mat{cc}{0 & g_j^\T}.
}
For convenience, we let
\eqnn{
I_{2d} = \mat{cccccc}{e_1 & \cdots & e_d & \dot{e}_1 & \cdots & \dot{e}_d},
}
i.e. for all $j\in\set{1,\dots,d}$ 
we let $e_j\in\R^{2d}$ denote the $j$--th standard Euclidean basis vector and
let $\dot{e}_j = e_{d+j}\in\R^{2d}$ denote the $(d+j)$--th such vector; though a mild abuse of the ``dot'' (``$\dot{\hphantom{e}}$'') notation, this convention simplifies the subsequent exposition.
Now for all $j\in\set{1,\dots,d}$ we have
\eqnn{
f_j &= F_{\word(j)}(0,\nu\ones) = \mat{c}{\nu\ones \\ \alpha\ones - \beta\nu\ones - \delta \word(j)}\in\R^{2d},\\
g_j^\T &= Dh_{\eta(j)}(0,\nu\ones) = e_{\eta(j)}^\T\in\R^{1\times 2d},
}
and hence for any $i\in\set{1,\dots,d}$ we have
$g_i^\T f_j = \nu$.
These simplifications yield for all $j\in\set{1,\dots,d-1}$
\eqnn{
D&\vphi_{\word(j+1)}^+(0,(0,\nu\ones)) 
 D\vphi_{\word(j)}^+(0,(0,\nu\ones)) \\
& = \brak{I_{2d+1} + \frac{1}{\nu}\mat{c}{1 \\ -f_{j+1}} \mat{cc}{0 & g_{j+1}^\T}}
\brak{I_{2d+1} + \frac{1}{\nu}\mat{c}{1 \\ -f_{j}} \mat{cc}{0 & g_{j}^\T}} \\
& = I_{2d+1} + \frac{1}{\nu}\mat{c}{1 \\ -f_{j+1}} \mat{cc}{0 & g_{j+1}^\T}
+ \frac{1}{\nu}\mat{c}{0 \\ f_{j+1} - f_j} \mat{cc}{0 & g_{j}^\T}.
}
Noting that $f_{j+1} - f_j = -2\delta \dot{e}_{\eta(j)}$,
we conclude 
that $e_{\eta(i)}^\T(f_{j+1}-f_j) = 0$
for all $i\in\set{j+1,\dots,d}$.
This implies that
\eqnn{\label{eqn:salt2}
\Xi_\word &= \prod_{j=1}^d D\vphi_{\word(j)}^+(0,(0,\nu\ones)) \\
& = I_{2d+1} + \frac{1}{\nu}\mat{c}{1 \\ -f_{d}} \mat{cc}{0 & e_{\eta(d)}^\T}
- \frac{2\delta}{\nu}\sum_{j=1}^{d-1} \mat{c}{0 \\ \dot{e}_{\eta(j)}} \mat{cc}{0 & e_{\eta(j)}^\T}.
}
Noting for any $f\in\R^{2d}$ with $0_{2d} = 0\cdot\ones_{2d}\in\R^{2d}$ that
\eqnn{
\mat{cc}{0 & 0_{2d}^\T \\ f & I_{2d}} = I_{2d+1} + \mat{c}{-1 \\ f}\mat{cc}{1 & 0_{2d}^\T},
}
defining $f^+ = F_{+\ones}(0,\nu\ones)$ and noting that $f^+ - f_d = -2\delta \dot{e}_{\eta(d)}$ we have
\eqnn{
\mat{cc}{0 & 0_{2d}^\T \\ f^+ & I_{2d}}
\prod_{j=1}^d & D\vphi_{\word(j)}^+(0,(0,\nu\ones)) \\
= &
\, I_{2d+1} 
+ \mat{c}{-1 \\ f^+}\mat{cc}{1 & 0_{2d}^\T} 
 + \frac{1}{\nu}\mat{c}{0 \\ f^+ - f_d}\mat{cc}{0 & e_{\eta(d)}^\T} \\ 
& - \frac{2\delta}{\nu}\sum_{j=1}^{d-1} \mat{c}{0 \\ \dot{e}_{\eta(j)}} \mat{cc}{0 & e_{\eta(j)}^\T} \\
= & \, I_{2d+1} 
- \frac{2\delta}{\nu}\sum_{j=1}^{d} \mat{c}{0 \\ \dot{e}_{\eta(j)}} \mat{cc}{0 & e_{\eta(j)}^\T}
+ \mat{c}{-1 \\ f^+}\mat{cc}{1 & 0_{2d}^\T}.
}
Finally, defining $f^- = F_{-\ones}(0,\nu\ones)$ we have
\eqnn{
D&\phi_\word(0,(0,\nu\ones))\\
& = \mat{cc}{0 & 0_{2d}^\T \\ f^+ & I_{2d}}
\prod_{j=1}^d D\vphi_{\word(j)}^+(0,(0,\nu\ones)) 
\mat{cc}{0 & 0_{2d}^\T \\ f^- & I_{2d}} \\
& = 
I_{2d+1} 
- \frac{2\delta}{\nu}\mat{ccc}{0 & 0_{d}^\T & 0_{d}^\T \\ 0_{d} & 0 & 0 \\ 0_{d} & I_{d} & 0} 
+ 
\brak{
\mat{c}{-1 \\ f^-} -
\mat{c}{0 \\ 0_{d} \\ 2\delta\ones}
}
\mat{ccc}{1 & 0_{d}^\T & 0_{d}^\T}.
}
Restricting to the derivative with respect to state, we find for all words $\word\in\Words$ that
\eqnn{\label{eqn:Dxphiword2}
D_x\phi_\word(0,(0,\nu\ones)) = \mat{cc}{I_d & 0 \\ -\frac{2\delta}{\nu} I_d & I_d}.
}
Thus the piecewise--differentiable flow $\phi$ is $C^1$ with respect to state at $(0,(0,\nu\ones))$, and~\eqref{eqn:sync2:salt} follows.
Restricting instead to the derivative with respect to time, we find as expected that
\eqnn{\label{eqn:Dtphiword2}
D_t\phi_\word(0,(0,\nu\ones)) = 
f^- -
\mat{c}{0_{d} \\ 2\delta\ones}
= f^+ = F_{+\ones}(0,\nu\ones).
}

\iftoggle{siads}
{
\bibliographystyle{siam-sburden}
\bibliography{multi}
}
{
\hypersetup{linkcolor=blue}
\printbibliography
}

\pagebreak
\sect{Global Piecewise--Differentiable Flow}\label{app:flow}
This section contains a proof of {\corflow} that consists of a straightforward adaptation of the proof of~\cite[Theorem~9.12]{Lee2012} obtained by replacing all instances of the modifier ``smooth'' with ``piecewise--smooth''.

\begin{lemma}[Translation Lemma]\label{lem:trans}
Let $D\subset\R^d$ be open, $F\in \ECrD$, $J\subset\R$ be an interval, and $\xi:J\into D$ an integral curve for $F$.
For any $b\in\R$, the curve $\ha{\xi}:\ha{J}\into D$ defined by $\ha{\xi}(t) = \xi(t+b)$ is also an integral curve for $F$, where $\ha{J} = \set{t : t+b\in J}$.
\end{lemma}

\pf{
Clearly $\ha{\xi}\in PC^r(\ha{J},D)$, whence the fundamental theorem of calculus~\cite[Proposition~3.1.1]{Scholtes2012} in conjunction with Lemma~\ref{lem:Dflow} implies $\ha{\xi}$ is an integral curve for $F$.
}

\begin{theorem}[Fundamental Theorem on Flows]
\label{thm:fund}
If $F\in \ECrD$, then there exists a unique maximal flow $\phi\in PC^r(\e{F},D)$ for $F$.
This flow has the following properties:
\begin{enumerate}
\item[(a)] For each $x\in D$, the curve $\phi^x:\e{F}^x\into D$ is the unique maximal integral curve of $F$ starting at $x$.
\item[(b)] If $s\in\e{F}^x$, then $\e{F}^{\phi(s,x)} = \e{F}^x - s = \set{t - s : t\in\e{F}^x}$.
\item[(c)] For each $t\in\R$, the set $D_t = \set{x\in D : (t,x)\in\e{F}}$ is open in $D$ and $\phi_t:D_t\into D_{-t}$ is a piecewise--$C^r$ homeomorphism with inverse $\phi_{-t}$.
\end{enumerate}
\end{theorem}

\pf{
This proof is a straightforward adaptation of the proof of Theorem~9.12 in~\cite{Lee2012}.

{\thmflow} shows that there exists an integral curve for $F$ starting at each point $x\in D$.
Suppose $\xi,\td{\xi}:J\into D$ are two integral curves for $F$ defined on the same open interval $J$ such that $\xi(t_0) = \td{\xi}(t_0)$ for some $t_0\in J$.
Let $S = \set{s\in J : \xi(s) = \td{\xi}(s)}$.
Clearly $S\ne\emptyset$ since $t_0\in S$, and $S$ is closed in $J$ by continuity of integral curves.
On the other hand, suppose $t_1\in S$.
Applying {\thmflow} near $x = \xi(t_1)$, we see that there exists an interval $t_1\in I\subset\R$ such that $\xi|_I = \td{\xi}|_I$.
This implies $S$ is open in $J$.
Since $J$ is connected, $S = J$, which implies $\xi|_J = \td{\xi}|_J$.
Thus any two integral curves that agree at one point agree on their common domain.

For each $x\in D$, let $\e{F}^x$ be the union of all domains of integral curves for $F$ originating at $x$ at time $0$.
Define $\phi^x : \e{F}^x\into D$ by letting $\phi^x(t) = \xi(t)$, where $\xi$ is any integral curve starting at $x$ and defined on an open interval containing $0$ and $t$.
Since all such integral curves agree at $t$ by the argument above, $\phi^x$ is well--defined, and is obviously the unique maximal integral curve starting at $p$.

Now let $\e{F} = \set{(t,x)\in\R\times D : t\in\e{F}^x}$ and define $\phi:\e{F}\into D$ by $\phi(t,x) = \phi^x(t)$.
We also write $\phi_t(x) = \phi(t,x)$.
By definition, $\phi$ satisfies property (a) in the statement of the fundamental theorem: for each $x\in D$, $\phi^x$ is the unique maximal integral curve for $F$ starting at $x$.
To verify the group laws, fix any $x\in D$ and $s\in\e{F}^x$, and write $y = \phi(s,x) = \phi^x(s)$.
The curve $\xi:(\e{F}^x - s)\into D$ defined by $\xi(t) = \phi^x(t+s)$ starts at $y$, and Lemma~\ref{lem:trans} shows that $\xi$ is an integral curve for $F$.
Since $\phi$ is a function, $\xi$ agrees with $\phi^y$ on their common domain, which is equivalent to 
\eqnn{\label{eqn:grp}
\forall s\in\e{F}^x, t\in\e{F}^{\phi(s,x)}: (s+t\in\e{F}^x)\implies\paren{\phi(t,\phi(s,x)) = \phi(t+s,x)}.
}
The fact that $\phi(0,x) = x$ for all $x\in D$ is obvious.
By maximality of $\phi^x$, the domain of $\xi$ cannot be larger than $\e{F}^y$, which means that $\e{F}^x - s\subset\e{F}^y$.
Since $0\in\e{F}^x$, this implies $-s\in\e{F}^y$, and the group law~\eqref{eqn:grp} implies that $\phi^y(-s) = x$.
Applying the same argument with $(-s,y)$ in place of $(s,x)$, we find that $\e{F}^y + s\subset\e{F}^x$, which is the same as $\e{F}^y\subset\e{F}^x - s$.
This proves (b).

Next we show that $\e{F}$ is open in $\R\times D$ (so it is a flow domain) and that $\phi:\e{F}\into D$ is $PC^r$.
Define a subset $W\subset\e{F}$ as the set of all $(t,x)\in\e{F}$ such that $\phi$ is defined and $PC^r$ on a product neighborhood of $(t,x)$ of the form $J\times U\subset\e{F}$, where $J\subset\R$ is an open interval containing $0$ and $t$ and $U\subset D$ is a neighborhood of $x$.
Then $W$ is open in $\R\times D$, and the restriction $\phi|_W\in PC^r(W,D)$, so it suffices to show that $W = \e{F}$.
Suppose this is not the case.  Then there exists some point $(\tau,x_0)\in\e{F}\sm W$.
For simplicity, assume $\tau > 0$; the argument for $\tau < 0$ is similar (and can be obtained, for instance, by considering the flow for $-F$).

Let $t_0 = \inf\set{t\in\R : (t,x_0)\not\in W}$~\see{Fig.~9.6 in~\cite{Lee2012}}.
By {\thmflow}, $\phi$ is defined and $PC^r$ in some product neighborhood of $(0,x_0)$, so $t_0 > 0$.
Since $t_0 \le \tau$ and $\e{F}^{x_0}$ is an open interval containing $0$ and $\tau$, it follows that $t_0\in\e{F}^{x_0}$.
Let $y_0 = \phi^{x_0}(t_0)$.
By {\thmflow} again, there exists $\veps > 0$ and a neighborhood $U_0$ of $y_0$ such that $(-\veps,\veps)\times U_0\subset W$.
We will use the group law~\eqref{eqn:grp} to show that $\phi$ admits a $PC^r$ extension to a neighborhood of $(t_0,x_0)$, which contradicts our choice of $t_0$.

Choose some $t_1 < t_0$ such that $t_1 + \veps > t_0$ and $\phi^{x_0}(t_1)\in U_0$.
Since $t_1 < t_0$, we have $(t_1,x_0)\in W$, so there is a product neighborhood $(t_1-\delta,t_1+\delta)\times U_1\subset W$ for some $\delta > 0$.
By definition of $W$, this implies $\phi$ is defined and $PC^r$ on $[0,t_1+\delta)\times U_1$.
Because $\phi(t_1,x_0)\in U_0$, we can choose $U_1$ small enough that $\phi$ maps $\set{t_1}\times U_1$ into $U_0$.
Define $\td{\phi}:[0,t_1+\veps)\times U_1\into D$ by
\eqn{
\forall (t,x)\in [0,t_1+\veps)\times U_1 : \td{\phi}(t,x) = \pw{\phi_t(x),& x\in U_1,\ 0\le t < t_1, \\ \phi_{t - t_1}\circ\phi_{t_1}(x),& x\in U_1,\ t_1-\veps < t < t_1+\veps.}
}
The group law for $\phi$ guarantees that these definitions agree where they overlap, and our choices of $U_1$, $t_1$, and $\veps$ ensure that this defines a $PC^r$ map.
By Lemma~\ref{lem:trans}, each map $t\mapsto\td{\phi}(t,p)$ is an integral curve of $F$, so $\td{\phi}$ is a $PC^r$ extension of $\phi$ to a neighborhood of $(t_0,x_0)$, contradicting our choice of $t_0$.
This completes the proof that $W = \e{F}$.

Finally, we prove (c).
The fact that $D_t$ is open is an immediate consequence of the fact that $\e{F}$ is open.
From part (b) we deduce that
\eqn{
x\in D_t &\implies t\in\e{F}^x \implies \e{F}^{\phi_t(x)} = \e{F}^x - t \\
&\implies -t\in\e{F}^{\phi_t(x)} \implies \phi_t(x)\in D_{-t},
}
which shows that $\phi_t$ maps $D_t$ to $D_{-t}$.
Moreover, the group laws then show that $\phi_{-t}\circ\phi_t$ is equal to the identity on $D_t$.
Reversing the roles of $t$ and $-t$ shows that $\phi_t\circ\phi_{-t}$ is the identity on $D_{-t}$, which completes the proof.
}

\vspace{5cm}

\pagebreak
\sect{Perturbation of Differential Inclusions}\label{app:filippov}

In the proof of the perturbation results of \sct{pert}, we relied on a result due to Filippov.
For completeness, we reproduce the statement of the result we required.

\begin{assumption}[{\cite[Chapter~2, \S8, Theorem~1]{Filippov1988}}]\label{ass:filippov}
In the domain $\e{F}$ a set--valued function $F(t,x)$ satisfies the \emph{basic conditions} if for all $(t,x)\in \e{F}$ the set $F(t,x)$ is nonempty, bounded, closed, and convex, and furthermore the function $F$ is upper semicontinuous in $t,x$.
\end{assumption}

\noindent
Here, $\e{F}$ is understood to be a subset of $\R\times\R^d$, and $F$ is upper semicontinuous as a \emph{multifunction} $F:\e{F}\into\pset{\R^d}$~\cite[\S2.1]{Clarke1990}, i.e. for all $(t,x)\in\e{F}$, $\veps > 0$ there exists $\delta > 0$ such that
\eqnn{
\forall (s,y)\in (t,x) + B_\delta(0) : F(s,y) \subset F(t,x) + B_\veps(0)
}
where we adopt the usual notation in a Banach space $X$,
\eqnn{
\forall x\in X,\ B\subset X : x + B = \set{x + y : y\in B}.
}
As in~\cite[Chapter~2, \S8]{Filippov1988}, for any $\td{F}:\e{F}\into\pset{\R^d}$ we define the deviation $d_{\e{F}}(\td{F},F)$ as
\eqnn{
d_{\e{F}}(\td{F},F) = \inf\set{\delta > 0 \mid \forall (t,x)\in\e{F} : \td{F}(t,x)\subset \brak{\co{F\paren{t+ B_\delta(0), x+B_\delta(0)}}} + B_\delta(0)}
}
where for any $U\subset\R^d$ the set $\co{U}$ denotes the convex hull of points in $U$.

\begin{theorem}[{\cite[Chapter~2, \S8, Theorem~1]{Filippov1988}}]\label{thm:filippov}
Let $F(t,x)$ satisfy {\assinc} in the open domain $\e{F}$; $t_0\in [a,b]$, $(t_0,x_0)\in \e{F}$; let all the solutions of the problem
\eqnn{\label{eqn:filippov:1}
\dot{x}\in F(t,x),\ x(t_0) = x_0
}
exist for all $t\in [a,b]$ and their graphs lie in $\e{F}$.

Then for any $\veps > 0$ there exists a $\delta > 0$ such that for any $\td{t}_0\in[a,b]$, $\td{x}_0$ and $\td{F}(t,x)$ satisfying the conditions
\eqn{\label{eqn:filippov:2}
\abs{\td{t}_0 - t_0} \le \delta,\ 
\norm{\td{x}_0 - x_0} \le \delta,\ 
d_{\e{F}}(\td{F}, F) \le \delta
}
and {\assinc}, each solution of the problem
\eqnn{\label{eqn:filippov:3}
\dot{\td{x}}\in \td{F}(t,\td{x}),\ \td{x}(\td{t}_0) = \td{x}_0
}
exists for all $t\in [a,b]$ and differs from some solution of~\eqref{eqn:filippov:1} by not more than $\veps$.
\end{theorem}

\noindent
Here, a ``solution of the problem~\eqref{eqn:filippov:1}'' on the interval $[a,b]\subset\R$ is an absolutely continuous function $y:[a,b]\into \R^d$; 
its ``graph lies in $\e{F}$'' if $\set{(t,y(t)) : t\in [a,b]}\subset \e{F}$.

\end{document}